\documentclass[12pt]{article}
\usepackage{fullpage}%varioref,multicol

\usepackage{bm}
\usepackage[dvipsnames]{xcolor}
\usepackage{amsmath,amssymb}
\usepackage{amsthm}
\usepackage{wrapfig}
\usepackage{epsfig}
\usepackage{graphicx}
\usepackage{makeidx}
\usepackage{multicol}
\usepackage{tikz}
\usepackage{algorithm}
\usepackage{algpseudocode}
\usepackage{soul}

\usepackage{array,multirow,booktabs,bbm} % for fancier tables

\usepackage{algorithmicx}
\usepackage{algpseudocode}
\usepackage{epstopdf}
\usepackage{mathrsfs,mathtools}
\usepackage{xcolor}
\usepackage{enumitem}
\newcommand{\N}{{\cal N}}

\newcommand{\D}{{\cal D}}

\newcommand{\R}{I\hspace{-1ex}R}

\graphicspath{{./Figs/}}
\newcommand{\bs}[1]{\boldsymbol{#1}}
\newcommand{\bmu}{\bs{\mu}}

\newcommand{\bc}{\bs{c}}

\newcommand{\bx}{\bs{x}}

\newcommand{\calP}{\mathcal{P}}
\newcommand{\calD}{\mathcal{D}}

\newcommand{\calN}{\mathcal{N}}

\DeclareMathOperator*{\argmax}{argmax}
\DeclareMathOperator*{\argmin}{argmin}

\newtheorem{theorem}{Theorem}[section]

\newtheorem{lemma}{Lemma}[section]
\theoremstyle{remark}
\newtheorem{remark}{Remark}[section]

\begin{document}
\graphicspath{{Figs/}}

\title{An EIM-degradation free reduced basis method via over collocation and residual hyper reduction-based error estimation}

\author{
Yanlai Chen\footnote{Department of Mathematics, University of Massachusetts Dartmouth, 285 Old Westport Road, North Dartmouth, MA 02747, USA. Email: {\tt{yanlai.chen@umassd.edu}}.}, \, 
Sigal Gottlieb \footnote{Department of Mathematics, University of Massachusetts Dartmouth, 285 Old Westport Road, North Dartmouth, MA 02747, USA. Email: {\tt{sgottlieb@umassd.edu}}.}, \, 
Lijie Ji \footnote{School of Mathematical Sciences, Shanghai Jiao Tong University, Shanghai 200240, China. Email: {\tt sjtujidreamer@sjtu.edu.cn}. },\, 
Yvon Maday \footnote{
Sorbonne Universit\'e, Universit\'e Paris-Diderot SPC, CNRS, Laboratoire Jacques-Louis Lions, LJLL, F-75005 Paris and Institut Universitaire de France. Email: {\tt maday@ann.jussieu.fr}.} 
\footnote{%\newline ${   } \quad \ {  }$ 
L. Ji was partly supported by China Scholarship Council (CSC, No.201906230067) during the author's one year visit at University of Massachusetts, Dartmouth. Y. Chen and S. Gottlieb were partially supported by National Science Foundation grant DMS-1719698 and by AFOSR grant FA9550-18-1-0383. This material is based upon work supported by the National Science Foundation under Grant No. DMS-1439786 and by the Simons Foundation Grant No. 50736 while Y.~Chen and L.~Ji were in residence at the Institute for Computational and Experimental Research in Mathematics in Providence, RI, during the ``Model and dimension reduction in uncertain and dynamic systems" program.
}}

\date{\empty}

\maketitle

\begin{abstract}

The need for multiple interactive, real-time simulations using different parameter values has driven the design of fast numerical algorithms with certifiable accuracies. The reduced basis method (RBM) presents itself as such an option. RBM features a mathematically rigorous error estimator which drives the construction of a low-dimensional subspace. A surrogate solution is then sought in this low-dimensional space approximating the parameter-induced high fidelity solution manifold. However when the system is nonlinear or its parameter dependence nonaffine, this efficiency gain degrades tremendously, an inherent drawback of the application of the empirical interpolation method (EIM).

In this paper, we augment and extend the EIM approach as a direct solver, as opposed to an assistant, for solving nonlinear partial differential equations on the reduced level. The resulting method, called Reduced Over-Collocation method (ROC), is stable and capable of avoiding the efficiency degradation. Two critical ingredients of the scheme are collocation at about twice as many locations as the number of basis elements for the reduced approximation space, and an efficient  error estimator for the strategic building of the reduced solution space. The latter, the main contribution of this paper, results from an adaptive hyper reduction of the residuals for the reduced solution. Together, these two ingredients render the proposed R2-ROC scheme both offline- and online-efficient. A distinctive feature is that the efficiency degradation appearing in traditional RBM approaches that utilize EIM for nonlinear and nonaffine problems is circumvented, both in the offline and online stages. Numerical tests on different families of time-dependent and steady-state nonlinear problems demonstrate the high efficiency and accuracy of our R2-ROC and its superior stability performance.

\end{abstract}

%\newpage

\section{Introduction}

The need for highly efficient simulations of parametrized systems, often governed by partial differential equations (PDEs), is increasing in many areas of scientific and engineering applications.
In particular, the need for multiple interactive, real-time simulations using different parameter values has driven the design of 
 fast numerical algorithms with certifiable accuracies. The parameters involved may have a wide variety of physical meanings, including 
 boundary conditions, material properties, geometric settings, source properties etc. Moreover, the parameter dimensionality of the system may be high, the dependence of the system on the parameters may be complicated, the underlying systems may be nonlinear and their dependence on the parameters may be nonaffine.

To satisfy the need for  fast numerical algorithms with certifiable accuracies that can be used to efficiently compute multi-parametric systems,
the reduced basis method (RBM)  \cite{Quarteroni2015, HesthavenRozzaStammBook} was developed and proven effective. 
The RBM was introduced in the 1970s in the context of a nonlinear structure problem \cite{Almroth1978, noor1979reduced}. It has since been used in a wide variety of problems, including linear evolution equations \cite{HaasdonkOhlberger}, viscous Burgers equation \cite{veroy2003reduced}, the Navier-Stokes equations \cite{deparis2009reduced}, 
and harmonic Maxwell's equation \cite{chen2010certified, chen2012certified}, among many others. 
The success of RB methods depends on an offline-online decomposition process, where the costly process of basis selection and surrogate space construction 
are performed offline by a greedy algorithm, and an efficient online reconstruction using the reduced basis then provides orders-of-magnitude efficiency gain. 
The RBM is constructed so that the computational complexity of the online reduced solver is independent of the number of degrees of freedom of the 
high-fidelity approximation of the basis functions, and so can provide efficient real-time solutions.
Detailed reviews of the RBM approach  can be found in \cite{Rozza2008, Haasdonk2017Review} and \cite{Quarteroni2015, HesthavenRozzaStammBook}.

For mildly nonaffine terms and/or nonlinear equations, the  {\em Empirical Interpolation Method} (EIM) 
or its discrete version (DEIM) \cite{Barrault2004, grepl2007efficient, ChaturantabutSorensen2010, PeherstorferButnaruWillcoxBungartz2014} 
is typically used to  remove the online dependence on the cost of the high-fidelity approximation and  achieve the efficiency goals of RBM.
However, when the problem has a strong nonlinearity or nonaffinity, the EIM is often not feasible. 
Furthermore, even in cases where performing a (D)EIM is  feasible, it may not be efficient.
For example, in cases when the parameter dependence or the nonlinearity is complicated, 
the EIM decomposition may require many terms,  increasing the  
the online complexity  and potentially  severely degrading the reduced solver's online efficiency. To see this, consider  a simple  heat conduction problem with a nonaffine parameter dependence:
\[ -\nabla \cdot \left( a(x; \bmu) \nabla u\right) = f.\]
To handle the nonaffine parameter dependence, we first apply EIM to approximate the function $a(x; \bmu)$
using  a linear combination of $\bmu$-independent functions, 
\[ a(x; \bmu) \approx \sum_{q=1}^{Q_a}\theta_q(\bmu) a(x; \bmu^q).\]
Here  $\{\bmu^q\}_{q=1}^{Q_a}$ is an ensemble, typically chosen through a greedy procedure. 
The equation of interest is written in its weak form $a(u, v; \bmu) \coloneqq \left(a(x; \bmu) \nabla u, \nabla v\right) = (f, v)$, and 
the reduced-order solution space spanned by  the full order solutions  $\{\xi^1, \dots, \xi^N\}$ is identified during the offline learning stage.
Finally, the online reduced solver is assembled for each $\bmu$ with the corresponding stiffness matrix created via
\[
\left( a(\xi_i, \xi_j; \bmu) \right)_{i,j = 1}^N \coloneqq \left(a(x; \bmu) \nabla \xi_i, \nabla \xi_j\right)_{i,j = 1}^N = \sum_{q=1}^{Q_a} \theta_q(\bmu) \left(a(x; \bmu^q) \nabla \xi_i, \nabla \xi_j\right)_{i,j = 1}^N,
\]
where $\left(a(x; \bmu^q) \nabla \xi_i, \nabla \xi_j\right)_{i,j = 1}^N$ is computed offline. Notice that although the online solver is not dependent on the cost of the high fidelity approximations, its 
complexity is  linearly dependent on the number of EIM terms $Q_a$.
If $Q_a$ is large, this may lead to  substantial reductions in efficiency. 
When the model involves geometric parametrization (such as in  \cite{chen2012certified, BenaceurEhrlacherErnMeunier2018}), it has been observed that 
 $Q_a$ can be prohibitively large (i.e. much larger than the reduced space dimension $N$) 
 even if the more efficient matrix version of EIM \cite{negri2015efficient} is adopted. 
In this work, we present an approach  to mitigate this drawback of EIM. The proposed  
reduced residual (R2) based reduced over-collocation (ROC) method circumvents the 
efficiency degradation  (in both the offline and online stages) 
that plagues RBM approaches that utilize EIM for nonlinear and nonaffine problems.

\subsection{Overview of the  reduced-residual reduced over-collocation approach}
To overcome the limitations of the EIM framework, we adopt a collocation approach as we did in 
\cite{ChenGottlieb2013, ChenGottliebMaday}
rather than variational approaches such as Galerkin or Petrov-Galerkin \cite{BennerGugercinWillcox2015, CarlbergBouMoslehFarhat2011, CarlbergBaroneAntil2017}.
The reduced collocation method was developed in  \cite{ChenGottlieb2013}, which works well to circumvent the EIM degradation for the reduced solver but suffers from stability problems \cite{ChenGottliebMaday}. To mitigate the stability issue, we adopt an {\em over-collocation} approach 
where we collocate at approximately twice as many points as the dimension of the reduced order space. 
Half of these collocation points  interpolate the reduced solution, which is given by a linear combination of the basis elements. 
We choose the other collocation points based on a computational analysis of the reduced order residuals when these 
basis functions are identified during the offline procedure. These additional collocation points  ensure a good interpolation 
of the residual corresponding to an arbitrary parameter value when the reduced order space is used to solve the pPDE. 
However, over-collocation alone does not provide  online and offline efficiency, because the error estimators (which are 
critical for the construction of the reduced solution space), still require the application of EIM decomposition.

The challenge of computing error estimators  without requiring a costly EIM decomposition
is resolved by the second ingredient of our method. 
We propose an efficient alternative for guiding the strategic selection of parameter values to build the reduced solution space,
 an error estimator that is based on a reduced residual.
The key is a systematic and hierarchical reduction of the judiciously selected residuals.
In comparison, our previously proposed L1-based over collocation approach \cite{ChenJiNarayanXu2020} follows the guidance of the L1-norm of the coefficients, under a set of a Lagrangian basis, of the reduced basis solution. The proposed R2-based scheme sits on a more mathematically rigorous foundation.

Together, these two ingredients produce a  reduced residual reduced over-collocation method, which we will
refer to as the {\bf R2-ROC} method. This R2-ROC scheme is online efficient in the sense that the
online cost is independent of the number of degrees of freedom of the high-fidelity truth approximation,
and also avoids the efficiency degradation of a direct EIM approach for nonlinear and nonaffine problems. 
The  R2-ROC method is also highly efficient offline:  it requires minimal computation beyond the standard RBM 
cost of acquiring solution snapshots used to construct the reduced order space. 
Consequently, minimum number of simulations of the pPDE that make the offline preparation stage worthwhile 
(the ``break-even'' number of simulations)  is significantly smaller than traditional RBM, as we show in our numerical
examples for  the steady-state and time-dependent cases of the diffusion with cubic reaction and the viscous Burgers' equation.

The paper is organized as follows. In Section \ref{sec:R2-ROC-Alg}, we introduce and analyze our R2-ROC method. We also discuss the difference between our approach and several others. 
In Section \ref{num:final} we present numerical results for two test problems, the viscous Burgers' equation \cite{veroy2003reduced} 
and various nonlinear convection diffusion reaction equations.
For all our test problems, the R2-ROC is shown to have accuracy on par with the classical RBM, while demonstrating 
significantly improved efficiency due to the independence of the number of expansion terms resulting from the EIM decomposition. 
Finally, concluding remarks are drawn in Section \ref{sec:conclusion}.

\section{The Reduced over-collocation (ROC) method}
\label{sec:R2-ROC-Alg}
Let $\Omega \subset \mathbb{R}^{d}$ (for $d = 1, 2,$ or $3$) be a bounded physical domain on which
we define the problem
\begin{equation}
\calP(u(\bx; \bmu);\bmu)-f(\bx)=0,~\bx \in \Omega,
\label{eq:steadymodel}
\end{equation}
with appropriate boundary conditions. 
The term  $\calP$ is a parametric second order partial differential operator that may include linear 
and nonlinear functions of the solution $u(\bx; \bmu)$,  and its derivatives $\nabla u(\bx; \bmu)$, and $\Delta u(\bx; \bmu)$.
The  $p$-dimensional parameter  $\bmu$ lives in the space $\mathcal{D} \subset \mathbb{R}^{p}$.
The  solution $u(\bm{\mu}) \coloneqq u(\bx;\bmu)$ lives in a Hilbert space $H$;
for example, for a stationary Laplace problem, the space $H$ is typically the Sobolev space $H^1(\Omega)$. 
The R2-ROC method is designed to work for both steady state and time dependent problems, so we also consider 
 the transient problem
\begin{equation}
u_t +\calP(u(t, \bx; \bmu);\bmu)-f(\bx)=0,~\bx \in \Omega,
\label{eq:timemodel}
\end{equation}
with appropriate boundary (and initial) conditions. We first   focus on developing the algorithm for steady state problems \eqref{eq:steadymodel} 
and will then extend the algorithm to time dependent case \eqref{eq:timemodel} in Section \ref{sec:offline:time}.

\begin{table}[htbp]
  \begin{center}
  \resizebox{\textwidth}{!}{
    \renewcommand{\tabcolsep}{0.4cm}
    {\scriptsize
    \renewcommand{\arraystretch}{1.3}
    \renewcommand{\tabcolsep}{12pt}
    \begin{tabular}{@{}lp{0.8\textwidth}@{}}
      \toprule
      $\bmu = (\mu_1, \dots, \mu_p)$ & Parameter in $p$-dimensional parameter domain $\calD \subseteq \R^p$ \\
      $\Xi_{\rm{train}}$ & Parameter training set, a finite subset of $\mathcal{D}$ \\
      $u(\bmu)$ & Function-valued solution of a parameterized PDE on $\Omega \subset \mathbb{R}^{d}$\\
      $\calP(u(\bmu); \bmu)$ & A (nonlinear) PDE operator\\
 {$K$} & Number of finite difference intervals per direction of the physical domain\\
      $\mathcal{N} \approx K^d$ & Degrees of freedom (DoF) of a high-fidelity PDE discretization, the ``truth" solver \\ 
      $X^\calN$ & A size-$\calN$ (full) collocation grid\\
      $u^{\mathcal{N}}(\bmu)$ & Finite-dimensional truth solution\\
      $N$ & Number of reduced basis snapshots, $N \ll \mathcal{N}$\\
        $\bmu^j$ & ``Snapshot" parameter values, $j=1, \ldots, N$\\
      $\widehat{u}_n(\bmu)$ & Reduced basis solution in the $n$-dimensional RB space spanned by $\{u^{\mathcal{N}}(\bmu^1), \dots, u^{\mathcal{N}}(\bmu^n)\}$\\
      $e_n(\bmu)$ & Reduced basis solution error, equals $u^{\mathcal{N}}(\bmu) - \widehat{u}_n(\bmu)$ \\
      $\Delta_{{N}} \left(\bmu\right)$ & A residual-based error estimate (upper bound) for $\left\|e_N\left(\bmu\right)\right\|$ or an error/importance indicator\\
      $X^{N-1}_r=\{\bx^1_{**}, \dots, \bx^{N-1}_{**}\}$ & A size-$(N-1)$ reduced collocation grid, a subset of $X^\calN$ determined based on residuals\\
      $X^N_s = \{\bx^1_*, \dots, \bx^N_*\}$ & An additional size-$N$ reduced collocation grid, a subset of $X^\calN$ determined based on the solutions\\
      $X^M$ & A reduced collocation grid of size $M$ that is $X^{N-1}_r \cup X^N_s$\\
      $T$ & Final time for the time-dependent problems\\
      $\Delta t$ & Time stepsize for the time dependent problems\\
      $\calN_t$ & Total number of time levels, i.e. $\calN_t = \frac{T}{\Delta t}$\\
      $t^j$ & Time level $j$, $j=1, \ldots, \calN_t$\\
      $\epsilon_{\mathrm{tol}}$ & Error estimate stopping tolerance in greedy sweep \\
      \midrule
      Offline component & The pre-computation phase, where the reduced solver is trained using a greedy selection of snapshots from the solution space\\
      Online component & The process of solving the offline-trained reduced problem, yielding the reduced order solution.\\
    \bottomrule
    \end{tabular}
  }
    }
  \end{center}
\caption{Notation and terminology used throughout this article.}\label{tab:notation}
\end{table}
We  proceed by  discretizing  the equation \eqref{eq:steadymodel} by a high-fidelity scheme (known as a ``truth solver'' in the RB literature).  
We define  the discrete solution $u^\N(X^\N; \bmu)$ such that the equation
\begin{equation}
\calP_{\N}(u^\N(X^\N; \bmu);\bmu)-f(X^\N)=0,
\label{eq:pdesystem}
\end{equation}
is satisfied on a set of  $\N$ collocation points $X^\N \in \Omega$. 
With a slight abuse of notation, we let $\N$ denote the number of the degrees of freedom in the solver,
even though the $\N$ points in $X^\N$ might include, e.g. points on a Dirichlet boundary that are not free.

The truth approximation  $u^\N(X^\N; \bmu)$
is thus a discretization of the solution $u(\bmu)$ on the grid $X^\N$ so that the equation \eqref{eq:steadymodel} 
is enforced on a very refined discrete level. 
In \eqref{eq:pdesystem}, the terms  $\nabla u(X^\N; \bmu)$, and $\Delta u(X^\N; \bmu)$  are
approximated by the numerical discretizations $\nabla_h u(X^\N; \bmu)$, and $\Delta_h u(X^\N; \bmu)$,
where generally $h \propto \frac{1}{\sqrt[d]{\N}}$.
In this paper, we use a finite difference method (FDM)  to obtain this 
discretized equation.  However, the extension to point-wise schemes such as spectral collocation is obvious, 
and to finite element methods is possible.

We are now ready to describe the R2-ROC algorithm. We split the description into the online (in Section \ref{sec:online})
and offline (in Section  \ref{sec:offline_c}) components. However, specification of part of the online algorithm is postponed 
until the introduction of the ROC offline algorithm in Section \ref{sec:offline_c}, 
which repeatedly calls the online solver to construct a surrogate solution space. Analysis of the method is provided in Section \ref{sec:analysis}.
In Section \ref{sec:offline:time} we  present the extension to the time-dependent problems of the form  \eqref{eq:timemodel}.
The algorithms contain a great many terms with associated superscripts and subscripts. To help avoid confusion, 
we list all these terms and their meaning in Table \ref{tab:notation} for the readers' reference.

\subsection{Online algorithm}

\label{sec:online}

The online component of the R2-ROC is similar to the online component of the reduced collocation method described in our prior work \cite{ChenGottlieb2013}, 
with one critical difference: in the R2-ROC method we  use a larger number of collocation points than the number
of reduced basis snapshots. Moreover, the selection of these points now takes a closer account of the PDE we are solving. This {\em over-collocation} feature is an innovative approach that provides additional stabilization 
of the online solver, as we will observe in the numerical results.

We begin with $N$ selected parameters $\{\bmu^1, \dots, \bmu^N\}$, and the corresponding high fidelity truth approximations 
$\{ u_n \coloneqq u^\N(X^\N;\bmu^n), 1 \le n \le N\}$. We also have a set of  collocation points $ X^M $ $(M \, \ge N)$ 
formed from  a subset of $X^\N$,
\begin{align*}
  X^M = \{\bx_*^1, \dots, \bx_*^M\}, \quad \mbox{ with }\bx_*^j  \mbox{ having index } i_j \mbox{ in }  X^\N. 
\end{align*}
These three ingredients: the chosen parameters, the truth approximations, and the set of collocation points are all identified and computed in the 
{\em offline} phase and will be described in Section \ref{sec:offline_c}.

Note that we adopt the same notation for a function $u_n$ and  its discrete representation as a vector of its values at the grid points. 
These vectors $\{ u_n, 1 \le n \le N\}$ constitute the columns of basis matrix  $W_{n} \in \mathbb{R}^{\N \times n}$ for $n \in \{1, \dots, N\}$. 
Furthermore, we let $W_{n,M}$ denote the matrix of the corresponding reduced basis space on the set $X^M$,
\begin{align*}
W_{n,M} &= [u_{1}(X^M),  \ldots, u_{n}(X^M)] \in \mathbb{R}^{M \times n}, \quad \mbox{for } n = 1, \dots, N.\\
& =P_* W_{n},
\end{align*}
where the operator $P_* \in {\mathbb R}^{M \times \calN}$ is defined as,
\begin{align*}
P_* =\left[e_{i_1}, \cdots, e_{i_M}\right]^T,
\end{align*}
with $e_{i} \in \mathbb{R}^{\N \times 1}$ the $i$ th canonical unit vector.

We are now ready to describe the online algorithm. For any given parameter $\bmu$ we seek a reduced approximation of the solution 
$u(\bmu)$, denoted by $\widehat{u}_n(\bmu)$ and computed as a linear combination of the truth approximation ``snapshots'' contained in the 
reduced basis space $W_{n} $. This reduced basis solution satisfies
\[\widehat{u}_n(\bmu) = W_{n} \bc_n (\bmu),\] where 
the  coefficients $\bc_n (\bmu)$ must satisfy a reduced version of equation \eqref{eq:pdesystem}:
\begin{equation}
\calP_{\N}(W_{n} \bc_n (\bmu); \bmu) \approx f(X^\N).
\label{eq:reducedidea}
\end{equation}
Recall that the number of ``snapshots'' $n$ is (hopefully significantly) smaller than the degrees of freedom of the 
truth approximation $\calN$, so we have an over-determined system. In  \cite{ChenGottlieb2013} we dealt with this system by a 
Petrov Galerkin approach or collocation on $n$ points (which produced a square system). In this paper we propose something different, 
which is one of the distinctive features of our method. Indeed, the R2-ROC method proposes to solve the unknown coefficients $\bc_n(\bmu)$ by  minimizing  the residual of \eqref{eq:reducedidea} on the set of nodes $X^M$:
\[  P_* \left(\calP_{\N}(W_{n} \omega; \bmu)-f(X^\N)\right).\]
The problem is formulated as an optimization problem: 
\begin{align}
\bc_n(\bmu) = \argmin_{\omega \in \mathbb{R}^n}\parallel P_* \left(\calP_{\N}(W_{n} \omega; \bmu)-f(X^\N)\right)\parallel_{\mathbb{R}^M}.
\label{pde:reduced}
\end{align}
Note that this is a nonlinear system of equations for $\bc_n(\bmu)$. It can be solved using an iterative methods such as Newton's method or Picard iteration.

The calculation of 
\[ P_* \left( \calP_{\N}(W_{n} \omega; \bmu)-f(X^\N) \right) \]
relies on  the computation of
\begin{align*}
  \nabla_h\widehat{u}_n(\bmu) &=  P_* \left[\left(\nabla_h u_{1}\right),  \ldots, \left( \nabla_h u_{n}\right)\right] \bc_n(\bmu),\\
  \Delta_h\widehat{u}_n(\bmu)& =P_* \left[\left(\Delta_h u_{1}\right),  \ldots, \left( \Delta_h u_{n}\right)\right] \bc_n(\bmu).
\end{align*}
Notice that the differentiations $\nabla_h u_{j}$ and $\Delta_h u_{j}$ are computed accurately,  at a cost proportional of $\N$
and then projected to the reduced grid $X^M$. However, this step is performed offline, and  only a  matrix of dimension $M\times n$
is used in the online computation of $  \nabla_h\widehat{u}_n(\bmu)$ and $  \Delta_h\widehat{u}_n(\bmu)$ (due to linearity of $\nabla_h \cdot$ and $\Delta_h \cdot$) so the online cost remains low. 
The collocation approach  allows for solving this system with a cost only dependent on $M$ and $n$ even when $\calP_{\N}$ is nonlinear and nonaffine due to its point-wise evaluation nature.  
As a consequence, the online solver is independent of the degrees of freedom $\N$ of the underlying truth solver.  

In summary, the online procedure of the nonlinear solve for obtaining $\bc_n(\bmu)$ from equation \eqref{pde:reduced} involves: 
\begin{itemize}
\item [1)] realizing/updating $W_{n,M}\bc_n$, $\nabla_h(W_{n,M})\bc_n$,  and $\Delta_h(W_{n,M})\bc_n$  at each iteration, at a cost of  $O(M n)$ operations; 
\item [2)] calculating the forcing term $f(X^M)$, at a cost of $O(M)$ operations; and 
\item [3)] solving the reduced linear systems at each iteration of the nonlinear solve, at a cost of  $O(n^3)$ operations per iteration.
\end{itemize}

The next section describes the offline procedure in which  we select the  $N$ reduced basis parameters $\{\bmu^1, \dots, \bmu^N\}$
sequentially through a greedy algorithm.  Once a selected parameter $\bmu^j$ is determined,  we precompute as many 
 quantities as possible so that minimal update is performed at each iteration of the online iterative method. We also describe the choice of 
 the over-collocation points $X^M$ and analyze the resulting scheme in the next section.

\subsection{Offline algorithm}

\label{sec:offline_c}

In this section, we describe the offline procedure of the algorithm. There are three components here, and we describe each separately.

\subsubsection{A greedy algorithm}
\label{sec:L1greedy}

Reduced basis methods typically utilize  a greedy scheme to iteratively construct  the reduced basis space. The R2-ROC is no exception.
In this section we describe the procedure for  selecting the representative parameters $\bmu^1, \ldots, \bmu^N$ which comprise the reduced parameter space,
and the corresponding reduced basis space $W_N$. We utilize a greedy scheme to iteratively construct $W_N$ as follows:

We begin by selecting the first parameter $\bmu^1$ randomly from $\Xi_{\rm train}$ (a discretization of the parameter domain $\calD$) and 
we obtain its corresponding high-fidelity truth approximation $u^\mathcal{N}(\bmu^1)$ to
form a (one-dimensional) RB space given by the range of $W_1 = \left[u^{\mathcal N}(\bmu^1)\right]$. 
Now assume that we begin each iteration with a $n$-dimensional reduced parameter space and 
reduced basis space $W_n$ comprised of the corresponding truth approximations.
Next, we use the online procedure described above to obtain an RB approximation $\widehat{u}_{n}(\boldsymbol{\mu})$ for each parameter in $\Xi_{\rm train}$ 
and compute its error estimator $\Delta_n(\bmu)$. The $(n+1)$th parameter $\bmu^{n+1}$ is now selected using a 
greedy approach and the RB space augmented by
\begin{equation}
\label{eq:rbmgreedy}
  \bmu^{n+1} = \underset{\bm{\mu} \in \Xi_{\rm train}}{\argmax} \Delta_{n}(\bm{\mu}), \quad \quad W_{n+1} = \left[ W_n \;\; u^\calN(\bmu^{n+1})\right].
  \end{equation}

For this procedure to be efficient and accurate, the  greedy algorithm requires an efficiently-computable error estimate that quantifies the discrepancy 
between the $n$-dimensional surrogate solution  $\widehat{u}_n(\bmu)$ and the truth solution $u^\calN(\bmu)$.  We denote this error estimator 
 $\Delta_n(\bmu)$, it traditionally satisfies $\Delta_n(\bmu) \geq \left\| \widehat{u}_n(\bmu) - u^\calN(\bmu)\right\|$. The  error bound $\Delta_n$ is usually defined based on a residual-type {\em a posteriori} error estimate from the truth discretization. 
Mathematical rigor and implementational efficiency of this error estimate are crucial for the accuracy of the reduced basis solution and 
its efficiency gain over the truth approximation.  When $\calP(u; \bmu)$ is a linear operator, the Riesz representation theorem and a 
variational inequality imply that $\Delta_n$ can be taken as 
$$\Delta_n^R (\bmu) = \frac{\lVert f -
\mathcal{P}_\mathcal{N}(\widehat{u}_n; \bmu) \rVert_2} {\sqrt{\beta_{LB}(\bmu)}}, 
$$
which is a rigorous bound (with the $^R$-superscript denoting that it is based on the full residual). 
Here $\beta_{LB}(\bmu)$ is a lower bound for the smallest eigenvalue of ${P}_\mathcal{N}(\bmu)^T {P}_\mathcal{N}(\bmu)$ 
where ${P}_\mathcal{N}(\bmu)$ is the matrix corresponding to the discretized linear operator $\mathcal{P}_\mathcal{N}(\cdot; \bmu)$.

\subsubsection{An error estimator based on {\em Reduced Residual}}

For the general nonlinear equation, deriving the counterpart of this estimation is far from trivial. 
Moreover, even for linear equations, the robust evaluation of the residual norm in the numerator is delicate \cite{Casenave2014_M2AN, JiangChenNarayan2019}. 
Furthermore, we would also have to resort to an offline-online decomposition to retain efficiency which usually means application of EIM for nonlinear or nonaffine terms. 
This complication degrades, sometimes significantly \cite{BenaceurEhrlacherErnMeunier2018, negri2015efficient}, the online efficiency due to the large number of resulting 
EIM terms.  What exacerbates the situation further is that the (parameter-dependent) stability factor $\beta_{LB}(\bmu)$ must be calculated by a computationally efficient 
procedure such as the successive constraint method \cite{HuynhSCM, HKCHP}.  In this section we present our novel reduced-residual error estimator as an alternative that does not suffer from any of these challenges.

This alternative error estimator must be as efficient and effective  for the nonlinear and 
nonaffine problems as for the linear affine ones. We  present here our novel {\em reduced residual} based error estimator:
\begin{equation}
\Delta_n^{RR}(\bmu) = \lVert f - \mathcal{P}_\mathcal{N}(\widehat{u}_n; \bmu) \rVert_{L^\infty(X^M)}= \lVert P_*\left( f - \mathcal{P}_\mathcal{N}(\widehat{u}_n; \bmu)\right) \rVert_{L^\infty}.
\label{eq:deltaRR}
\end{equation}
Note that this residual is not being evaluated over the entire discrete mesh of the truth approximation, 
only a judiciously reduced portion of it.  It is therefore based on the {\em reduced residual}, giving the name of the method - R2-based reduced over collocation. 

The effectiveness of this error estimator is wholly dependent on the choice of the over-collocation set $X^M$, the topic of the next sub-section. 
Our analysis in Section \ref{sec:analysis} will show that first set of points of  $X^M$, denoted by $X_s^N$, ensures that, when the differential operator is linear, our reduced collocation solution recovers a specifically designed generalized empirical interpolant \cite{MagicPt_2009, MadayMulaTurinici_GEIMSIAM} of the truth approximation. 
The remaining part of $X^M$, denoted by $X_r^{N-1}$, is critical in maintaining the online-efficiency of $\Delta_n^{RR}$ in \eqref{eq:deltaRR} while providing a stable interpolating procedure for $f - \mathcal{P}_\mathcal{N}(\widehat{u}_n; \bmu)$ of any $\bmu \in \calD$ in the space of those at the greedy-selected $\bmu^n$'s, thus a mechanism to control $\lVert f - \mathcal{P}_\mathcal{N}(\widehat{u}_n; \bmu)  \rVert_{L^\infty(X^\N)}$ which is stronger than $\Delta_n^{RR}$, but not online-efficient.

In the numerical examples we demonstrate that this reduced residual error estimator is a reliable quantity to monitor when deciding which  
representative parameters $\bmu^1, \dots, \bmu^N$ will form the surrogate space. 
In addition, a further advantage of this error estimator over our previously proposed L1-ROC \cite{ChenJiNarayanXu2020} 
is that $\Delta_n^{RR}$ does decrease as $n$ increases. In fact, the numerical results seem to 
indicate the effective index is rather constant and small. 
Moreover, the calculation of $\Delta_n^{RR}$ is independent of $\calN$ while the traditional $\Delta_n^R$ is dependent on $\calN$. 
This difference leads to the dramatic efficiency gain of the R2-ROC, as we will numerically confirm in Section \ref{num:final}.

\begin{algorithm}[h]
\begin{algorithmic}[1]
\vspace{0.5ex}
\State Choose  $\bmu^1$ randomly from $\Xi_{\rm train}$, compute $u_1\coloneqq u^\N(X^\N;\bmu^1)$. 
\State Compute $\bx_*^1 = \argmax_x |\calP_\calN(u_1; \bmu^1)(\bx)|, \, \sigma_1(\cdot) = \sigma_{\bx_*^1}^{\bmu^1}(\cdot)$, define $\xi_1 = u_1 / \sigma_1(u_1)$. Let $i_1$ be the index of $\bx_*^1$ and $P_* = [e_{i_1}]^T$.
\State Initialize $m = n = 1, \, X^m = X^n_s =[\bx_*^1]$, $W_1 = \left\{\xi_1 \right\}, W_{1,m} = P_*W_1$, and $X_r^0 = \emptyset$. 
\State \mbox{\textbf{For}} $n = 2,\ldots, N$ 
\State $\quad\ \mbox{Solve }  \bc_{n-1} (\bmu) $ with $W_{n-1}, P_*$ and calculate $\Delta_{n-1}(\bmu)$  for all $\bmu \in \Xi_{\rm train}$.  
\State $\quad\ \mbox{Find } \bmu^{n} = \argmax_{\bmu \in \Xi_{\rm train}\backslash\left\{\bmu^i,i=1,\cdots,n-1\right\}}  \Delta_{n-1}(\bmu)$ and solve for $\xi_n \coloneqq u^\N(X^\N;\bmu^n)$. 
\State $\quad\ \mbox{Compute a generalized interpolatory residual for } \xi_n: \, \mbox{find } \{\alpha_j\} \mbox{ and let } \xi_n = \xi_n - \sum_{j=1}^{n-1}\alpha_j \xi_j$ \mbox{so that } $\sigma_i(\xi_n) = 0, \forall i \in \{1, \cdots, n-1\}$. %$\xi_n(X^{n-1}_s)=0$.
\State $\quad\ \mbox{Find }$ $\bx_*^{n} = \argmax_x |\calP_\calN(\xi_{n}; \bmu^{n})(\bx)|, \, \sigma_{n}(\cdot) = \sigma_{\bx_*^{n}}^{\bmu^{n}}(\cdot)$, $\xi_n=\xi_n/ \sigma_n(\xi_n)$, and let $X^n_s = X^{n-1}_s \cup \{\bx_*^n\}$, and $i_1$ be the index of $\bx_*^n$.
\State $\quad\ \mbox{Form the full residual vector } r_{n-1} =\calP_{\N}(\widehat{u}_{n-1}(\bmu^n);\bmu^n)-f(X^\N)$ and compute its interpolatory residual: $\, \mbox{find } \{\alpha_j\} \mbox{ and let } r_{n-1} = r_{n-1} - \sum_{j=1}^{n-2}\alpha_j r_j$ \mbox{so that } $r_{n-1}(X_r^{n-2})=0$. Find $\bx^{n-1}_{**}=\argmax_{\bx \in X^\N/\left\{X^m,\bx_*^n\right\}} |r_{n-1}(\bx)|$. Let $r_{n-1}=r_{n-1}/ r_{n-1}(\bx^{n-1}_{**})$, and $X^{n-1}_r =X^{n-2}_r \cup \{\bx^{n-1}_{**}\}$ and $i_2$ is the index of $\bx_{**}^{n-1}$.
\State $\quad\ \mbox{Update } W_{ n} = \{W_{n-1}, \xi_n\}, m=2n- 1, X^m = X^n_s \cup X_r^{n-1}, P_*= P_*\cup [e_{i_1},e_{i_2}]^T$.
\State \mbox{\textbf{End For}}
\end{algorithmic}
\caption{R2-ROC: construction of $W_N$ and the  collocation set $X^{2N - 1} = X^N_s \cup X^{N-1}_r$.} 
\label{alg:c:plus:offline2}
\end{algorithm}

\subsubsection{Construction of the reduced over-collocation set $X^M$}

\label{sec:offline_oc}

There is one piece that we have left until all the other ingredients are in place:
we are now ready to describe how to  determine the reduced collocation set $X^M$ that are needed in 
both the online and offline algorithms. The reduced collocation set $X^M$ is comprised of collocation points
that are selected using two different approaches. We will describe these as two different sets.
The first set of point, denoted by $X^N_s$,  consists of the maximizers from a Generalized EIM (GEIM) procedure \cite{MagicPt_2009, MadayMulaTurinici_GEIMSIAM} that is tailored to our setting. 

Indeed, the differentiating feature is the GEIM interpolating functional which we define as follows for any admissible function $v(\bx)$, any $\bx \in \Omega$, and $\bmu \in \calD$
\begin{equation}
\sigma_{\bx}^{\bmu}(v) = \calP_\calN(v; \bmu)(\bx).
\label{eq:functional}
\end{equation}
When $u_1 = u^\N(\cdot; \bmu^1)$ is calculated, we identify the first collocation point and the corresponding functional as 
\begin{equation}
\bx_*^1 = \argmax_x |\calP_\calN(u_1; \bmu^1)(\bx)|, \quad \sigma_1(\cdot) = \sigma_{\bx_*^1}^{\bmu^1}(\cdot)
\label{eq:point_functional_1}
\end{equation}
 and our first collocation basis $q_1 = \frac{u_1}{\sigma_1(u_1)}$ and $B_{11} = \sigma_1(q_1)$.

We then proceeds as follows. For $n = 1, 2, \cdots\cdots$, when $\bmu^{n+1}$ is identified by the greedy algorithm and $u_{n+1}$ obtained, we solve $\left\{ \alpha_{n+1,i}\right\}_{i=1}^n$ such that 
\[
\sigma_i\left(q_{n+1} \triangleq u_{n+1} - \sum_{i=1}^n \alpha_{n+1,i} q_i\right) = 0, \quad \forall \,\, i \in \{1, \cdots, n\}.
\]
We then augment the collocation points and functionals
\begin{equation}
\bx_*^{n+1} = \argmax_x |\calP_\calN(q_{n+1}; \bmu^{n+1})(\bx)|, \quad \sigma_{n+1}(\cdot) = \sigma_{\bx_*^{n+1}}^{\bmu^{n+1}}(\cdot),
\label{eq:point_functional_2}
\end{equation}
 and define $q_{n+1} = \frac{q_{n+1}}{\sigma_{n+1}(q_{n+1})}$. 
Lastly, we expand the matrix $B$ by a column and a row via $B_{ij} = \sigma_i(q_j)$ when $i$ or $j$ equals $n+1$. We finally define the first collocation set $X^N_s = \{\bx_*^1, \cdots, \bx_*^N\}$. 

The second set of points is chosen due to a recognition of the 
importance of controlling the residuals of the PDE when solving the equations.
In order to control the PDE  residuals, we must represent them well on the reduced grid. 
For that purpose, we introduce a second set of points, called $X^{N-1}_r$, which are chosen
using a greedy algorithm aiming to control the residual. To examine the  residual of the RB solution at the chosen $\bmu^{n}$ when only $n-1$ basis elements are used, 
we compute the residual vectors
\begin{equation}
r_{n-1}^{\bmu^n} =\calP_{\N}(\widehat{u}_{n-1}(\bmu^n);\bmu^n)-f(X^\N), \quad n \in \{2, \dots, N\}.
\label{eq:offlineresidual}
\end{equation}
These residuals are a basis that can be used to  interpolate the residual at any other point $\bmu$ in the domain,
and so we need to identify the collocation (or interpolation) points on which  this residual basis $\{ r_n\}$ best represents all
possible residuals.
For this reason, we take these $N-1$ residual vectors and perform an EIM procedure on them. 
The $N-1$ maximizers from this procedure form the second set which is denoted $X^{N-1}_r$.

The choice of the set of over-collocation points $X^M$ includes the points in $X^N_s$ and $X^{N-1}_r$, and so
we use $M= 2N-1$  collocation points.  
Note that the first basis function has no accompanying residual vector \eqref{eq:offlineresidual}, 
so that from the second onward there are two collocation points selected whenever a new parameter is identified by the greedy algorithm. 
We are now ready to outline the entire R2-ROC method in Algorithm \ref{alg:c:plus:offline2}.

\begin{remark}
The reduced collocation approach in \cite{ChenGottlieb2013} is a specialization that takes 
$\sigma_i(\cdot)$ to be identities, $M = N$, $X^M = X^N_s$. The resulting $M=N$ reduced scheme can be unstable particularly 
when high accuracy (i.e. large $N$) of the reduced solution is desired. 
It can be resolved in special cases by an analytical preconditioning approach \cite{ChenGottliebMaday}. 
The second obvious choice of $X^M$ is to append $X^{N-1}_r$ with one more point such as the maximizer of the first basis. 
Numerical tests (not reported in this paper) also reveal instability of this scheme. 
\end{remark}

%\newpage

\subsubsection{Analysis of the R2-ROC method}

\label{sec:analysis}

With the basis $W_N = \{q_1, \cdots, q_N\}$, the collocation points and the functionals as built in \eqref{eq:point_functional_1} and \eqref{eq:point_functional_2}, we are ready to define a Generalized Empirical Interpolation \cite{MagicPt_2009, MadayMulaTurinici_GEIMSIAM} operator for any admissible function $v(x)$
\begin{equation}
I_N^{\rm RC} [v] = \sum_{i = 1}^N \beta_i^N q_i(x) \mbox{ such that } \sigma_i(I_N^{\rm RC}[v]) = \sigma_i(v) \,\, \forall \, i \in \{1, \cdots, N\}.
\label{eq:RC_GEIM}
\end{equation}

\begin{lemma}
When the differential operator $\calP(u(\bx; \bmu); \bmu)$ is linear (with respect to $u$), the matrix $B$ is lower triangular with unitary diagonal and we have that $I_N[v] = v$ for $v \in {\rm span}\{q_1, \cdots, q_N\}$. If, in addition, the collocation points are taken as $X^M = \{\bx_*^1, \cdots, \bx_*^N\}$, our reduced collocation solution coincides with the GEIM approximation $I_N^{\rm RC}[u^\calN(\bmu)]$ of the truth approximation $u^\calN(\mu)$ when $\bmu = \bmu^i$ for $i \in \{1, \cdots, N\}$. 
\end{lemma}
\begin{proof}
When $\calP$ is linear, we have $\sigma_i(q_i) = 1$ and $\sigma_i(q_j) = 0$ when $j > i$ by construction. Therefore, matrix $B$ is lower triangular with unitary diagonal.

When $v \in {\rm span}\{q_1, \cdots, q_N\}$, we have that 
\[
v = \sum_{i=1}^N d_i q_i.
\]
Since $\calP$ is linear, we have $\sigma_i(v) = \sum_{j=1}^N d_j \sigma_i(q_j)$ which means $\vec{\sigma v}= B \vec{d}$ where $\vec{\sigma v} = (\sigma_1(v) \cdots \sigma_N(v))^T$ and $\vec{d} = \{d_1, \cdots, d_N\}^T$. On the other hand, from \eqref{eq:RC_GEIM} we know that, if we assume $I_N^{\rm RC}[v] = \sum_{i=1}^N c_i q_i$, we have that
\begin{equation}
\left(
\begin{tabular}{c}
$c_1$\\
$\vdots$\\
$c_N$
\end{tabular}
\right)
=
B^{-1}
\left(
\begin{tabular}{c}
$\sigma_1(v)$\\
$\vdots$\\
$\sigma_N(v)$
\end{tabular}
\right)
=B^{-1} \vec{\sigma v}
\label{eq:INRC_Solve}
\end{equation}
Plugging $\vec{\sigma v}= B \vec{d}$ completes the first half of the proof.

To prove the second half, we note that the reduced collocation procedure amounts to requiring that
\[
\sigma_{\bx^i}^{\bmu} (\widehat{u}_n(\bmu)) = \sigma_{\bx^i}^{\bmu} (u^\calN(\bmu))
\]
When $\bmu = \bmu^i$, this is identical to the system determining the GEIM approximation \eqref{eq:RC_GEIM}. Given the fact that $B$ is invertible, we conclude that the reduced collocation solution is identical to the GEIM approximation.
\end{proof}

Regarding the error of the reduced solution, we can prove a standard result of interpolation-type. Toward that end, we define the Lebesgue constant
\[
\Lambda_N = \sup_{\bx \in \Omega} \sum_{i=1}^N \left |h_i^N(\bx)\right |
\]
where $\left\{ h_i^N(x) : i = 1, \cdots, N \right\}$ is the Lagrangian basis satisfying
\begin{equation}
(h_1^N \, h_2^N \, \cdots \, h_N^N) B = (q_1 \, q_2 \, \cdots \, q_N).
\label{eq:qh_transform}
\end{equation}
It is straightforward to check that we do have $\sigma_i(h_j) = \delta_{ij}$.

\begin{theorem}
When the differential operator $\calP$ is linear and the collocation points are taken as $X^M = \{\bx_*^1, \cdots, \bx_*^N\}$, our reduced collocation solution satisfies the following estimate.
\begin{equation}
\lVert \widehat{u} (\bmu)- I_N^{\rm RC}[\widehat{u}(\bmu) ]\rVert_{L^\infty} \le (1 + \Lambda_N) \inf_{v \in W_N} \lVert \widehat{u}(\bmu) - v\rVert_{L^\infty}
\end{equation}

\end{theorem}
\begin{proof}
For any $v \in W_N$, we have 
\begin{alignat*}{1}
\lVert \widehat{u} (\bmu)- I_N^{\rm RC}[\widehat{u}(\bmu) ]\rVert_{L^\infty} \le & \lVert \widehat{u} (\bmu)- v\rVert_{L^\infty}  + \lVert v- I_N^{\rm RC}[\widehat{u}(\bmu) ]\rVert_{L^\infty} \\
= & \lVert \widehat{u} (\bmu)- v\rVert_{L^\infty}  + \lVert I_N^{\rm RC}[v - \widehat{u}(\bmu) ]\rVert_{L^\infty} \\
\le & \left(1 + \sup_{w \in W_N}\frac{\lVert I_N^{RC}[w] \rVert_{L^\infty}}{\lVert w \rVert_{L^{\infty}}} \right)\lVert \widehat{u} (\bmu)- v\rVert_{L^\infty} 
\end{alignat*}
To complete the proof, we just need to show that $\sup_{w \in W_N}\frac{\lVert I_N^{RC}[w] \rVert_{L^\infty}}{\lVert w \rVert_{L^{\infty}}} = \Lambda_N$ which can be verified by recalling equations \eqref{eq:INRC_Solve} and \eqref{eq:qh_transform}.
\end{proof}

These results show that our first set of points $X_s^N$ ensures that, when the differential operator is linear, our reduced collocation solution recovers a specifically designed generalized empirical interpolant of the truth approximation. We finish our analysis by the following remark which indicates the significance of the second set of points $X_r^{N-1}$.
\begin{remark}
{\bf Functionality of the second set of points $X_r^{N-1}$:} We recall that the residual \eqref{eq:offlineresidual} of RB solution at $\bmu$ with $n$ bases is defined as
\[
r_{n}^{\bmu} =\calP_{\N}(\widehat{u}_{n}(\bmu);\bmu)-f(X^\N) = \calP_{\N}(\widehat{u}_{n}(\bmu);\bmu) - \calP_{\N}({u}^\N(\bmu);\bmu).
\]
Assuming we have the error-residual relation, we need to control $\lVert r_n^{\bmu} \rVert_{L^\infty(X^\N)}$ for any $\bmu$. The greedy algorithm with error estimator $\Delta_n^{RR}$ in \eqref{eq:deltaRR} means that we have, for any $\bmu$,
\[
\lVert r_n^{\bmu} \rVert_{L^\infty(X^M)} \le \lVert r_n^{\bmu^{n+1}} \rVert_{L^\infty(X^M)}, \mbox{ a weak version of } \lVert r_n^{\bmu} \rVert_{L^\infty(X^\N)} \le \lVert r_n^{\bmu^{n+1}} \rVert_{L^\infty(X^\N)}.
\]
The latter can be achieved by adopting a stronger $\Delta_n^{RR} \, ({\rm i.e.} = \lVert r_n^{\bmu} \rVert_{L^\infty(X^\N)})$ which would make the online complexity linearly dependent on $\calN$ (thus the algorithm not online efficient).  The choice of the second set of points $X_r^{N-1}$ is critical in maintaining the online-efficiency of $\Delta_n^{RR}$ in \eqref{eq:deltaRR} while providing a stable interpolating procedure for $ r_n^{\bmu} $ in the space ${\rm span}\{r_i^{\bmu^{i+1}}\}_{i=1}^n$.

Indeed, if we denote the EIM interpolation procedure of $\{r_i^{\bmu^{i+1}}\}_{i=1}^n$ in Algorithm \ref{alg:c:plus:offline2} by $J_n$, we have that
\begin{alignat*}{1}
\lVert r_n^{\bmu} \rVert_{L^\infty(X^\N)} & \le \lVert J_n[r_n^{\bmu}] \rVert_{L^\infty(X^\N)} + \lVert r_n^{\bmu} -  J_n[r_n^{\bmu}]\rVert_{L^\infty(X^\N)}\\
& \le \lVert J_n[r_n^{\bmu}] \rVert_{L^\infty(X^\N)} + (1 + \Lambda_n^r) \inf_{v} \lVert r_n^{\bmu} - v \rVert_{L^\infty(X^\N)},
\end{alignat*}
where $\Lambda_n^r$ is the Lebesgue constant of  $\{r_i^{\bmu^{i+1}}\}_{i=1}^n$.Consider that the (classical) greedy algorithm adopted by EIM/GEIM has the tendency of minimizing the Lebesgue constant $\Lambda_n^r$ \cite{MadayMulaPateraYano2015}. We therefore conclude that the EIM procedure of $r_n^{\bmu}$ by $\{r_i^{\bmu^{i+1}}\}_{i=1}^n$  via their EIM points in $X^M$ is effective in generating the RB space and the online solver.
\end{remark}

%\newpage

\subsection{Extension of R2-ROC for time dependent problems}
\label{sec:offline:time}

Given the reduced space $W_n$ and the collocation set $X^M$, 
the semi-discretized R2-ROC solver remains identical to the steady-state case for the time-dependent problem \eqref{eq:timemodel}. That is, we seek the reduced approximation of the solution for any given parameter $\bmu$  in the form of
$$\widehat{u}_n(\bmu, t) = W_{n} \bc_n (\bmu, t).$$ 
The unknown coefficients $\bc_n(\bmu, t) \in {\mathbb R}^{n \times 1}$ is obtained by solving the following optimization problem:
\begin{align}
\bc_n(\bmu, t) = \argmin_{\omega}\parallel P_* \left(W_{n} \omega_t + \calP_{\N}(W_{n} \omega; \bmu)-f(X^\N)\right)\parallel_{\mathbb{R}^M}.
\label{pde:reduced_t}
\end{align}
For full dicretization, our R2-ROC aligns with the parameter-time greedy framework \cite{grepl2007efficient, grepl2005posteriori}, as opposed to POD \cite{Kunisch_Volkwein_POD, nguyen2008efficient} or POD-greedy \cite{grepl2012}. We first denote the (full) set of temporal nodes as ${\mathcal T}_f \coloneqq \{t_i: \, i =0, \cdots, \calN_t\}$ with $t_0$ being the initial time and $\calN_t =T/ \Delta t$ where $\Delta t$ is the temporal step-size. 
We also denote a {\em reduced} set of temporal nodes by ${\mathcal T}_r$ that starts from the empty set and is gradually enriched in the greedy algorithm.

We next describe the R2-based error estimator needed by the greedy algorithm for the time-dependent case as follows. It is extended from the steady-state version \eqref{eq:deltaRR}. 
Indeed, for each $\bmu \in \Xi_{\rm train}$, after its corresponding (reduced) solver of \eqref{pde:reduced_t} with $n$ bases is performed and the {\em reduced} solution $\widehat{u}_n(\bmu, t)$ obtained, we define
\begin{align}
\Delta_n^{RR_t}(\bmu) \coloneqq \sum_{t \in {\mathcal T}_f}  \varepsilon^{RR}(t; \bmu)  \,\, \mbox{ with } \,\, \varepsilon^{RR}(t; \bmu) \coloneqq \lVert P_*r_n(t; \bmu)\rVert_\infty.
\label{eq:delta_t}
\end{align}
Here, $r_n(t; \bmu)\in \mathbb{R}^{\N \times 1}$ denotes the {\em full} residual for $\widehat{u}_n(\bmu, t)$, and $P_*r_n(t; \bmu) \in  \mathbb{R}^{M \times 1}$ its reduced version. 
\begin{remark}
We emphasize that: 
1) The distinctive feature of our scheme, in comparison to e.g. \cite{grepl2005posteriori}, is that we only consider the {\em reduced} residuals, i.e. the residual sampled at our over collocation points; And 2) We automatically have online efficiency, without EIM, when evaluating the error estimator. This is made possible thanks to the collocation framework. 
\end{remark}

We are now ready to describe our greedy algorithm. To initiate the reduced solver construction we start with a deterministically or randomly chosen $\bmu^1$ (similar to the steady-state case) and invoke the truth solver to obtain the snapshots $\{u^\calN(t_i, x; \bmu^1)\}_{i=0}^{\calN_t}$. ${\mathcal T}_r$ is initiated by the time instant when the corresponding snapshot has the largest variation. That is,
\[
  {\mathcal T}_r = \{t_{\bmu^1}^{1}\} \mbox{ where } t_{\bmu^1}^{1} = \argmax_{t \in {\mathcal T}_f} \left( \max_{x \in X^\N} u^\calN(t, x; \bmu^1)- \min_{x \in X^\N} u^\calN(t, x; \bmu^1) \right).
\]
The RB space $W_1$ is initiated with $u^\calN(t_{\bmu^1}^{1}, x; \bmu^1)$. The (first) collocation point is set to be the special GEIM point of this first basis, 
\[
  \bx_\ast^1 = \argmax_{x \in X^\N} |\calP_\calN(u^\calN(t_{\bmu^1}^{1}, x; \bmu^1); \bmu^1)(x)|.
\]
We note that the corresponding collocation functional $\sigma_{\bx}^{\bmu}$, defined in \eqref{eq:functional} for the steady-state case, should be understood as its extension to the time-dependent version
\[
\sigma_{t,\bx}^{\bmu}(v) = v_t + \calP(v(t,x; \bmu); \bmu)(t, \bx).
\]
However, for brevity of notation, we still write it as $\sigma_{\bx}^{\bmu}$ whenever the accompanying $t$ is clear.

\begin{algorithm}[h]
\begin{algorithmic}[1]
\vspace{0.5ex}
\State Choose $\bmu^1$, and set $k_{\bmu^1} = 1$ the first temporal node to be $t_{\bmu^1}^{k_{\bmu^1}} = \argmax_{t \in {\mathcal T}_f} \left( \max_x u^\calN(t, x; \bmu^1)- \min_x u^\calN(t, x; \bmu^1) \right)$. Define $\xi_1\coloneqq u^\mathcal{N}(t_{\bmu^1}^{k_{\bmu^1}}, X^\N;\bmu^1)$. 
\State Find $\bx_*^1=\argmax_{\bx \in X^\N} |\calP_\calN(\xi_1; \bmu^1)|\, \sigma_1(\cdot) = \sigma_{\bx_*^1}^{\bmu^1}(\cdot)$, and let $P_* = [e_{i_1}]^T$, where $i_1$ is the index of $\bx_*^1$.
\State Initialize $m = n = 1, \, X^m = X^n_s =\{\bx_*^1\}$, $W_1 = \left\{\xi_1 \right\}, W_{1,m} = P_*W_1$, and $X_r^0 = \emptyset$.  %\\[0.5ex]
\State  \mbox{\textbf{For}} $n = 2,\ldots, N$  %\\%[0.5ex]
\State  \quad Solve   the reduced problem for $\bc_{n-1} (\bmu, t_k) $. 
\State  $\quad\ \mbox{Find } \bmu^n = \argmax_{\bmu \in \Xi_{\rm train}}  \Delta_{n-1}^{RR_t}(\bmu)$, and a new temporal node $t_{\bmu^n}^{k_{\bmu^n}} = \arg \max_{t \in {\mathcal T}_f} {\varepsilon^{RR}(t;\bmu^n)}$. %\\%[.5ex]
\State  $\quad\ \mbox{Solve }  \xi_n=u^\mathcal{N} (t_{\bmu^n}^{k_{\bmu^n}}, X^\N;\bmu^n)$. %\\%[0.5ex]
\State  $\quad\ \mbox{Compute a generalized interpolatory residual for } \xi_n: \, \mbox{find } \{\alpha_j\} \mbox{ and let } \xi_n = \xi_n - \sum_{j=1}^{n-1}\alpha_j \xi_j$ \mbox{so that } $\sigma_i(\xi_n)=0$ for $i \in \{1, \cdots, n-1\}$. Find $\bx_*^n=\argmax_{\bx \in X^\N/X^m} |\calP_\calN(\xi_n; \bmu^n)(\bx)|, \sigma_n(\cdot) = \sigma_{x_*^n}^{\bmu^n}(\cdot)$, $\xi_n=\xi_n/\sigma_n(\xi_n)$. Let $X^n_s = X^{n-1}_s \cup \{\bx_*^n\}$, and $i_1$ be the index of $\bx_*^n$.%\\[0.5ex]
\State  $\quad\ \mbox{Form the full residual vector } r_{n-1} = \left(\widehat{u}_{n-1}\right)_t(t_{\bmu^n}^{k_{\bmu^n}};\bmu^n) + \calP_{\N}(X^\N, \widehat{u}_{n-1}(t_{\bmu^n}^{k_{\bmu^n}};\bmu^n);\bmu^n)-f(X^\N, t_{\bmu^n}^{k_{\bmu^n}})$. 
Compute an interpolatory residual $r_{n-1}: \mbox{find } \{\alpha_j\} \mbox{ and let } r_{n-1} = r_{n-1} - \sum_{j=1}^{n-2}\alpha_j r_j$ \mbox{so that } $r_{n-1}(X^{n-2}_r)=0$. Find $\bx_{**}^n=\argmax_{\bx \in X^\N/\left\{X^m,\bx_{*}^n\right\}} |r_{n-1}|$.Let  $r_{n-1}=r_{n-1}/ r_{n-1}(\bx_{**}^n)$, and $X^{n-1}_r =X^{n-2}_r \cup \{\bx_{**}^n\}$. $i_2$ is the index of $\bx_{**}^n$.%\\%[0.5ex]
\State  $\quad\ \mbox{Update } W_{n} = \{W_{n-1},\xi_n\}, m=2n- 1, X^m = X^n_s \cup X_r^{n-1},P_* = [P_*; \,\, (e_{i_1})^T; \, \, (e_{i_2})^T]$. %\\%[0.5ex]
\State  \mbox{\textbf{End For}}
\end{algorithmic}
\caption{R2-ROC algorithm for time dependent problems}
\label{alg:c:plus:offline:time}
\end{algorithm}
Once these ingredients are in place with the first pair $(\bmu^1, t^1_{\bmu^1})$ determined, we can solve the reduced problem \eqref{pde:reduced_t} for every $\bmu \in \Xi _{\rm train}$ (with a one-dimensional RB space $W_1$).  Similar to the traditional greedy algorithm, the next step is to determine the subsequent $(\bmu, t)$ pairs. Our greedy algorithm, as seen in Algorithm \ref{alg:c:plus:offline:time}, manifests itself in the following three aspects:
\begin{itemize}
\item {\bf Greedy in $\bmu$:} 
Our greedy choice for the $\bmu$-component of the $(\bmu, t)$ pair is through maximizing $\Delta_n^{RR_t}(\bmu)$ over the training set $\Xi _{\rm train}$:
\[
  \bmu^{n+1} = \argmax_{\bmu \in \Xi _{\rm train}} \Delta_n^{RR_t}(\bmu).
\]
\item {\bf Greedy in $t$:} 
Given the greedy choice $\bmu^{n+1}$ and the {\em reduced} solution $\widehat{u}_n(\bmu^{n+1}, t) = W_{n} \bc_n (\bmu^{n+1}, t)$ for all time levels $t \in {\mathcal T}_f$, the greedy $t$-choice is given by 
\begin{align}
t_{\bmu^{n+1}}^{k_{\bmu^{n+1}}} \coloneqq \argmax_{t \in {\mathcal T}_f}\left\{\varepsilon^{RR}( t;\bmu) \coloneqq {\lVert P_* r_n (t;\bmu^{n+1})\lVert}_\infty\right\}, \mbox{ and } {\mathcal T}_r = {\mathcal T}_r \bigcup \{t_{\bmu^{n+1}}^{k_{\bmu^{n+1}}}\}.
\label{eq:reducedresidual}
\end{align}
\item {\bf $X^M$ expansion:} With the greedy choice $(\bmu^{n+1}, t_{\bmu^{n+1}}^{k_{\bmu^{n+1}}})$, we solve for the truth approximations $u(t, X^\N; \bmu^{n+1})$ for $t \le t_{\bmu^{n+1}}^{k_{\bmu^{n+1}}}$. The expansion of $X^M$ by two more colocation points, with one from the GEIM procedure of the solution $u(t_{\bmu^{n+1}}^{k_{\bmu^{n+1}}}, X^\N; \bmu^{n+1})$ with the particular functionals $\{\sigma_{\bx_*^i}^{\bmu^i}(\cdot)\}_{i=1}^n$, and the other from that of the residual $r_n(t_{\bmu^{n+1}}^{k_{\bmu^{n+1}}}; \bmu^{n+1})$, is identical to the steady state case. 
\end{itemize}

\begin{remark}
Here, $k_{\bmu^{n+1}} \ge 1$ is introduced to accommodate the possibility that multiple temporal nodes might be selected for the same $\bmu$  (at different rounds of the greedy algorithm). 
We note in particular that, consistent with typical greedy scheme, we choose one (as opposed to multiple) maximizer in \eqref{eq:reducedresidual}. However, as we proceed with building up the reduced solution space, the same $\bmu$ (and a different temporal node) may be chosen by the greedy algorithm at a later step due to the lack of resolution of its corresponding temporal history. 
\end{remark}

\subsection{Other related techniques}
\label{sec:compare}

There are other model reduction techniques for both steady and trasient problems such as Proper Orthogonal Decomposition (POD) \cite{Kunisch_Volkwein_POD}, system-theoretic approaches including balanced truncation, moment matching or Hankel norm approximation \cite{BennerGugercinWillcox2015}. RBM differentiates itself, particularly for parametric problems, by featuring rigorous {\em a posteriori} error estimations, the resulting greedy algorithm, and the ability to compute the theoretically smallest number of full order solutions dictated by the Kolmogorov $n$-width of the solution manifold. Nonlinear problems bring some additional challenges, mainly in that a high dimensional reconstruction of the surrogate quantities is often needed each time the nonlinearity is evaluated. Sampling-based approximation techniques were developed to mitigate the resulting loss of efficiency. They include the Empirical Interpolation Method and its discrete variants \cite{Barrault2004, grepl2007efficient, ChaturantabutSorensen2010, PeherstorferButnaruWillcoxBungartz2014} and  Hyper-Reduction \cite{ryckelynck2005priori, Ryckelynck2009, CARLBERG2013623} which are known to be equivalent to DEIM under certain conditions \cite{Fritzen, Dimitriu2017}. Other approaches include POD coupled with ``the best interpolation points'' approach  \cite{nguyen2008efficient, galbally2010non}, Gappy-POD \cite{Everson1995}, Missing Point Estimation (MPE) \cite{Astrid2008MissingPoint} or Gauss–Newton with approximated Tensors (GNAT) \cite{CarlbergBouMoslehFarhat2011, CARLBERG2013623}. Most of these methods work by first identifying a subset of the important features of the nonlinear function, and then constructing an approximation of the full solution based solely on an evaluation of these few components.

The R2-ROC method presented in this paper can be viewed as adopting hyper reduction for {\em reduced} residual minimization. Indeed, instead of enforcing that the full residual is small in either a weak or strong formulation, we identify its selected entries and ensure that an accurate evaluation of the residual on that subset is small. This is not the first time this type of idea is explored. For example, \cite{Astrid2008MissingPoint, astrid2004fast} uses a collocation of the original equations based on missing point interpolation and is followed by a Galerkin projection. The authors in \cite{ryckelynck2005priori} obtain the solution snapshots and collocation points through an adaptive algorithm in the finite element framework. It was also applied to nonlinear dynamical systems with randomly chosen collocation points \cite{bos2004accelerating}. 
However, the proposed R2-ROC differs from these existing works. The first distinctive feature is that the  basis functions and collocation points  are determined hierarchically via a greedy algorithm guided by reduced residual minimization problems that gradually increase in size. 
It tailors the Generalized EIM procedure \cite{MagicPt_2009, MadayMulaTurinici_GEIMSIAM} to our setting via a set of carefully designed interpolating functionals. 
In comparison, the existing approaches obtain basis functions through POD-type techniques and then compute the whole set of collocation  points  all at once. The second distinctive feature is that the only step during the offline process that depends on the full order model is when we calculate a new high fidelity basis.

\section{Numerical results}
\label{num:final}

We test the R2-ROC method on nonlinear steady-state and time-dependent problems, respectively in Sections \ref{numerics:steady} and \ref{numerics:timedep}. The particular equations include the classical viscous Burgers' equation and nonlinear convection diffusion reaction equations. 

\subsection{R2-ROC for steady-state nonlinear problems}
\label{numerics:steady}
In this Section, we report the test results of R2-ROC on steady-state problems while comparing it with benchmark algorithms.

\subsubsection{Viscous Burgers' equation}
First, we test it on the one-dimensional (viscous) Burgers' equation,
  \begin{equation}
 \begin{split}
 u u_x & = \bmu u_{xx},\\
 u(x=-1) & = 1, \,\,\,\, u(x=1) = -1.
 \end{split}
 \label{eq:burgers}
\end{equation} 
Here the viscosity parameter $\bmu$ varies on the interval $\calD = [0.05,1]$. The computational domain $[-1,1]$ is divided uniformly into $\calN+1$ intervals with grid points denoted by 
\[
\{x_0, x_1, \dots, x_{\calN+1}\}.
\]
With $h = \frac{2}{\calN+1}$, the following  finite difference discretization based on the conservative form of equation \eqref{eq:burgers}, $\left( \frac{u^2}{2} \right)_x - \bmu u_{xx} = 0$, is then used 
\begin{equation}
\frac{u_{i+1}^2 -u_{i-1}^2}{4 h}- \bmu \frac{u_{i-1} -2u_i +u_{i+1}}{h^2}=0, \quad i \in \{1, \dots, \calN\}. 
\label{eq:conservative}
\end{equation}
This leads to a nonlinear truth solver of size $\calN$. The parameter domain $\calD$ is sampled 50 times logarithmically spaced to form the training set for the Offline procedure.  We test our method on a subset of $\Xi_{\rm test}$ of $\calD$ that does not intesect with the training set $\Xi_{\rm train}$. 
We compute the relative errors $E(n)$ over all $\bmu$ in $\Xi_{\rm test}$ of the RB solution $\widehat{u}_n(\bmu)$ in comparison to the high fidelity truth approximation. That is,
\begin{equation}
E(n) =  \max_{\bmu \in \Xi_{\rm test}}\left\{\frac{\| u(\bmu) - \widehat{u}_n(\bmu)\|_\infty}{\|u\|_{L^\infty(\Xi_{\rm test}, L^\infty(\Omega))}}\right\}
\label{eq:error:steadyburger}
\end{equation}
where 
\[
||u||_{L^\infty(\Xi_{\rm test}, L^\infty(\Omega))} = \max_{\bmu \in \Xi_{\rm test}} \|u(\bmu) \|_\infty.
\]

Error curves and the distribution of the first $N=10$ selected parameters with $\calN = 100$ are showed in Figure \ref{fig:steadyBurger}. It shows a clear exponential convergence as $n$ increases and a concentration of the selected $\mu$ values toward the lower end of the parameter domain. We note that the distributions of chosen parameters between the traditional RBM and the new R2-ROC are very much similar which underscores the reliability of our proposed approach. 
\begin{figure}[!htb]
\centering
\includegraphics[width=0.32\textwidth]{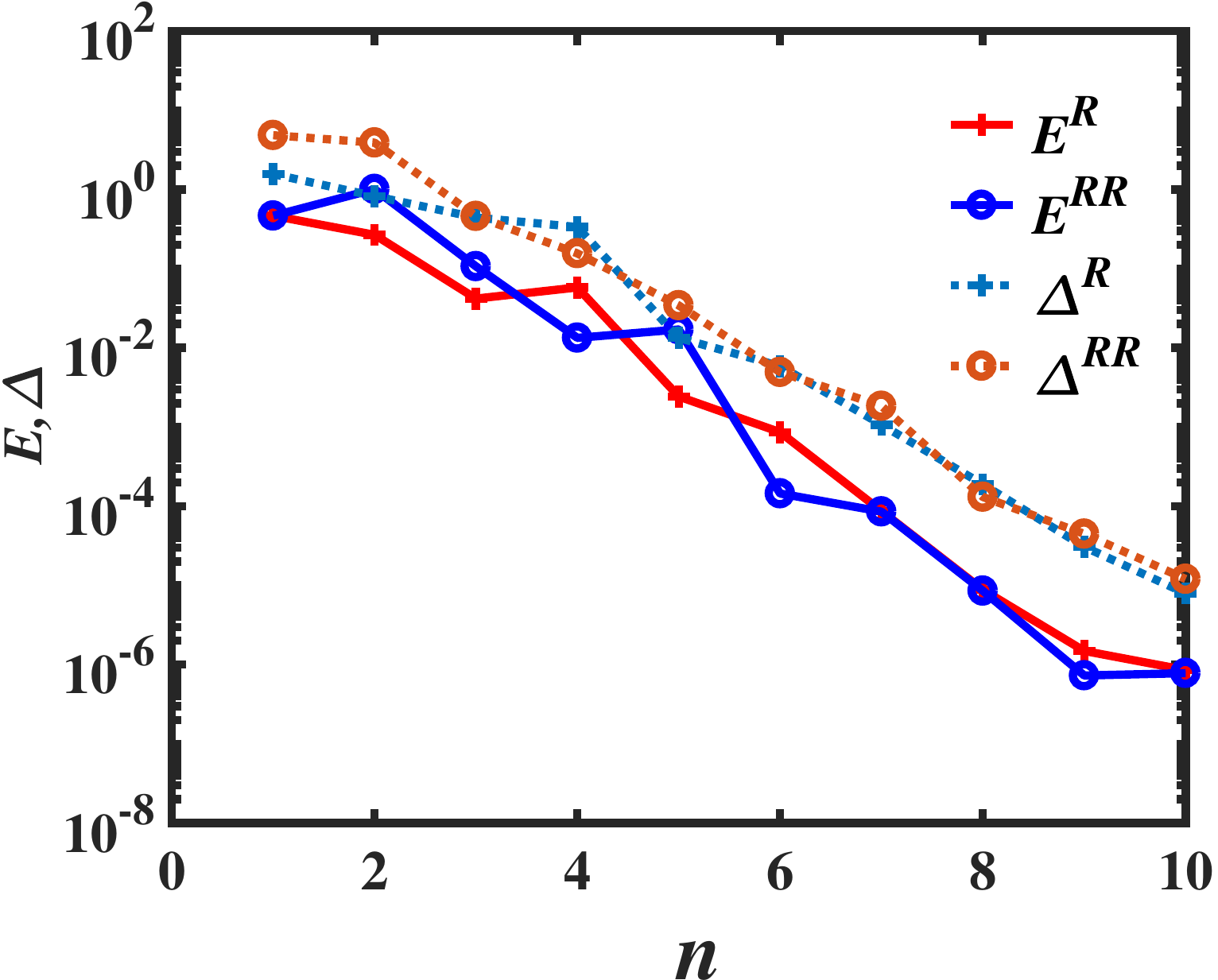}
\includegraphics[width=0.33\textwidth]{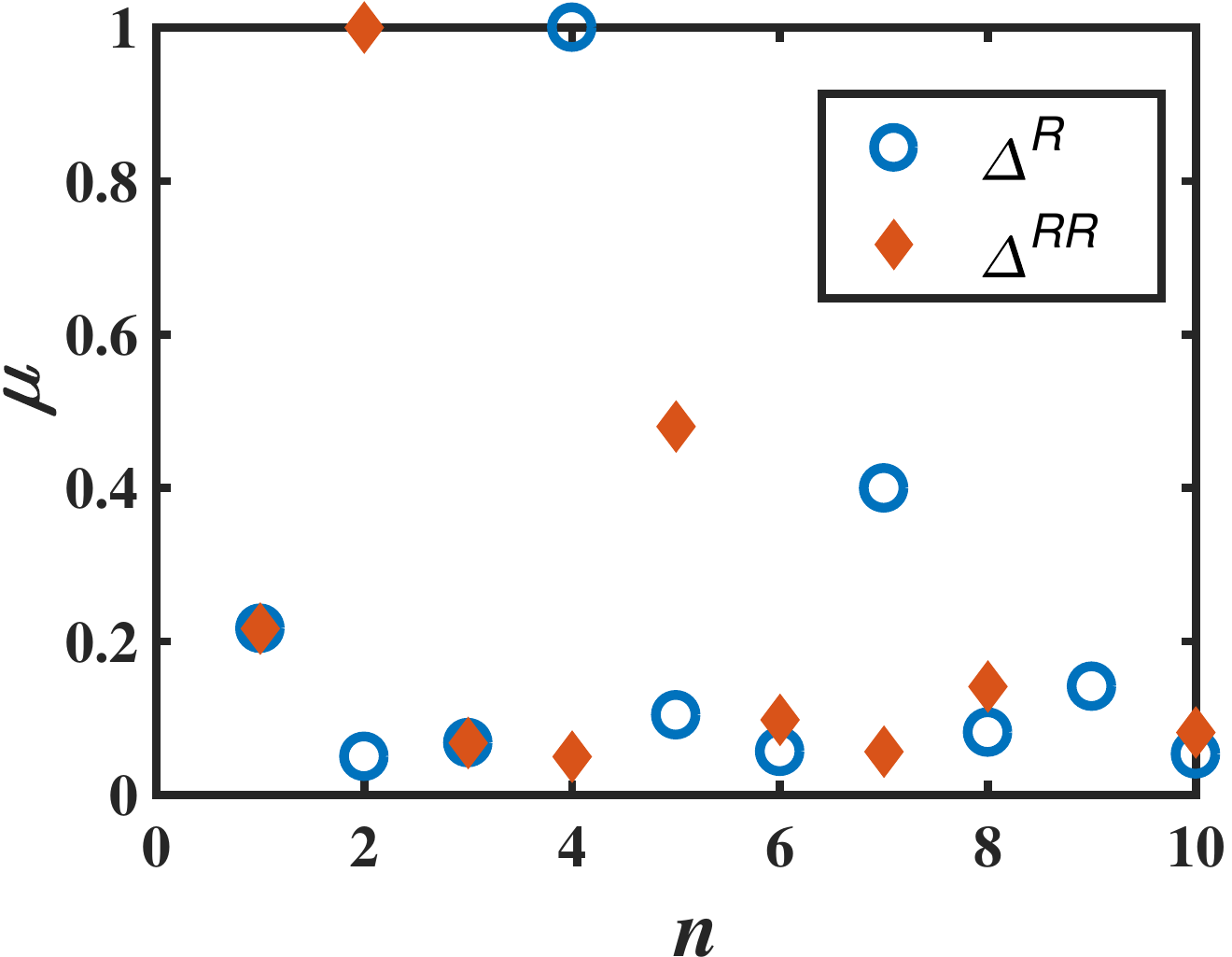}
\includegraphics[width=0.33\textwidth]{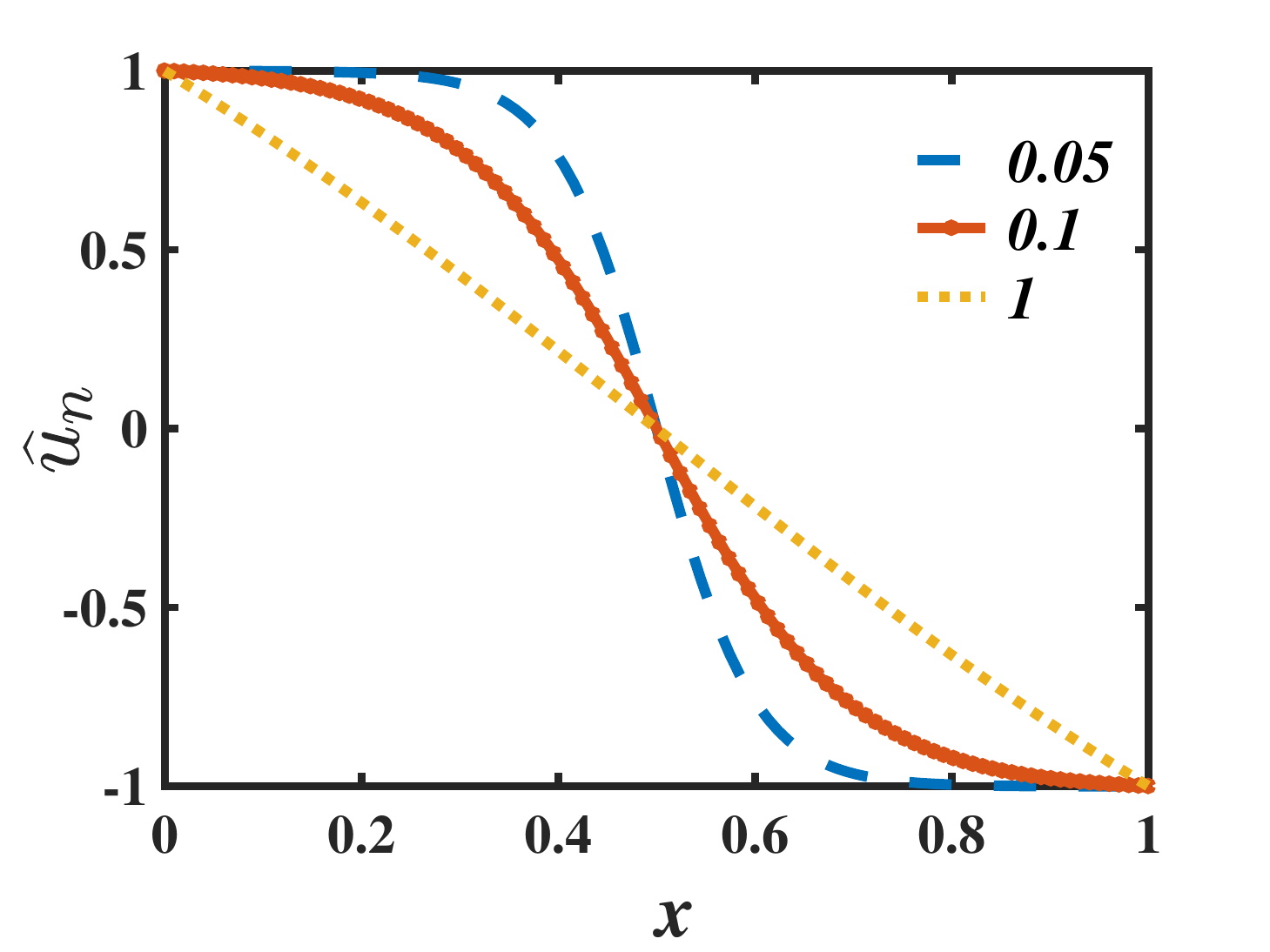}
\caption{Steady viscous Burgers' result. (Left) Histories of convergence for the error and error estimator for the traditional residual-based RBM and proposed R2-ROC. Here, $E^R$ and $E^{RR}$ refer to the $E(n)$ in \eqref{eq:error:steadyburger} with the reduced solution $\widehat{u}_n$ constructed by following the residual-based error estimator $\Delta^R$ and R2-based error estimator $\Delta^{RR}$, respectively. (Middle) Distribution of selected parameters $\bmu^n$, using estimator $\Delta^R$ and $\Delta^{RR}$, as a function of $n$. (Right) Sample RB solutions at three parameter values.}
\label{fig:steadyBurger}
\end{figure}

\subsubsection{Nonlinear reaction diffusion equations}

 Here we consider the following cubic reaction diffusion,
\begin{equation}\label{eq:CRD}
\begin{split}
-\mu_2 \Delta u +u{(u-\mu_1)}^2 & = f(\bx) \mbox{ in } \Omega := [-1,1]\times [-1,1],\\
u & = 0 \mbox{ on } \partial \Omega.
\end{split}
\end{equation}
We take $f(\bx)=100\sin(2\pi x_1)\cos(2\pi x_2)$, and $\D$ is set to be $[0.2,5]\times [0.2,2]$, and discretized by a $128 \times 64$ uniform tensorial grid. Denoting the step size along the $\mu_1$ direction by $h_1$, and the other by $h_2$, the training set and test set are given by 
\begin{align*}
  \Xi_{\rm train} &=  (0.2:4h_1:5) \times (0.2:4h_2:2), \\
  \Xi_{\rm test} &=   ((0.2+2h_1):4h_1:(5-2h_1)) \times ((0.2+2h_2):4h_2:(2-2h_2)),
\end{align*}
  where $(a:h:b)$ denotes an equidistant mesh over $[a,b]$ with stepsize $h$. The nonlinear solver, based on the $5$-point stencil with $\sqrt{\calN}$ interior points at each direction of $\Omega$, for the {high fidelity truth approximation} linearizes, at the  $(\ell+1)^{\rm th}$ iteration, the equation according to
\begin{equation}
-\mu_2 \Delta u^{(\ell+1)} +g'(u^{(\ell)})u^{(\ell+1)}=g'(u^{(\ell)})u^{(\ell)}-g(u^{(\ell)},\mu_1)+f(\bx)
\label{Operator2}
\end{equation}
where $g(u;\mu_1)=u{(u-\mu_1)}^2$.

Relative errors of the RB solution $E(n)$ with $K = \sqrt{\calN} = 400$ are displayed in Figure \ref{2relativeerror} top left showing steady exponential convergence for R2-ROC that is on par with the traditional RBM.
The set of selected parameters are shown in Figure \ref{2relativeerror} top middle, while the collocation points are shown on the bottom row. We note again that the distributions of chosen parameters between the traditional residual-based scheme and the more nascent R2-based scheme are quite similar for this example underscoring the reliability of R2-ROC.

Lastly, we showcase the vast saving of the offline time for the R2-ROC approaches. Toward that end, the comparison in cumulative computation time for the traditional residual-based RBM, R2-ROC, and the {high fidelity truth approximations} is shown in Figure \ref{2relativeerror} top right. 
The initial nonzero start of the R2-ROC is the amount of its offline time. 
We observe that the ``break-even'' number of runs for R2-ROC is much smaller than that of the traditional RBM. 
The difference in this ``break-even'' point is because the overhead cost, devoted to calculating $\Delta_n^{RR}$ (for R2-ROC), is significantly less than that for $\Delta_n^R$.  The latter involves (an offline-online decomposition of) the calculation of the full residual norm while the former only requires obtaining an $N\times 1$ vector and evaluating the residual at numbers of points proportional to the reduced space dimension. 
Though R2-ROC has a much more efficient offline procedure than the  residual-based ROC, their online time for any new parameter is comparable, see Table~\ref{time2}.  
The results also confirm that time consumption of the online ROC methods is independent of $\sqrt{\calN}$, that is the method is online efficient. 
Here in the table, we present the online calculation time for the different algorithms in two different parameter regimes. The first regime is when $\mu_1$ is large and $\mu_2$ small, in particular we choose $\mu_1=4.55, \mu_2=0.42$. The second regime has the relative sizes reversed. The reduced solver requires $27$ iterations for the nonlinear system in the first regime, while only requiring $8$ iterations in the second regime leading to the noticeable difference in the full-order time consumption. It also means that the speedup factor of R2-ROC varies. But they range  between $3000 \sim 12000$ when $\sqrt{\calN}=400,\, 800$.

\begin{figure}[!htb]
\centering
\includegraphics[width=0.32\textwidth]{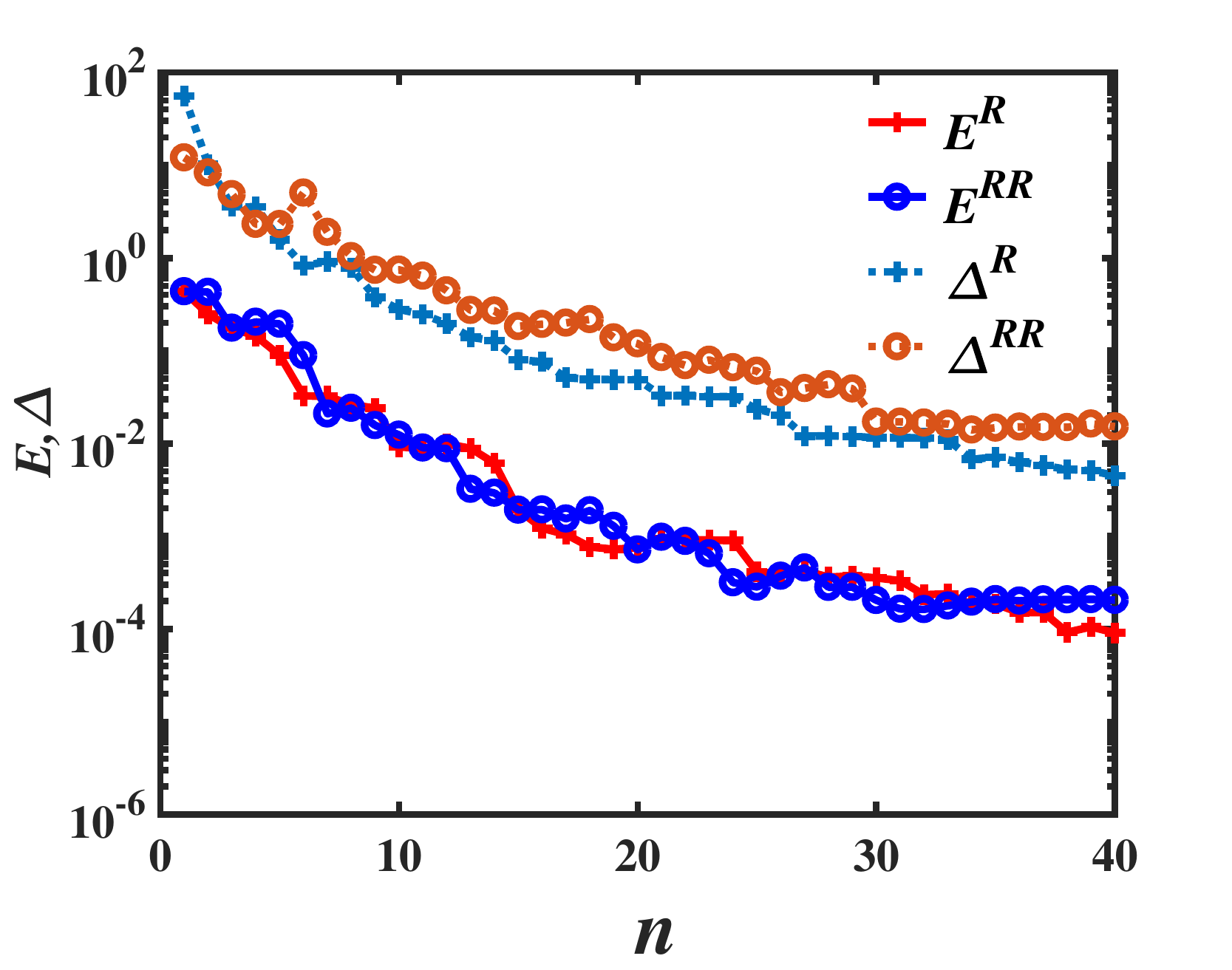}
\includegraphics[width=0.33\textwidth]{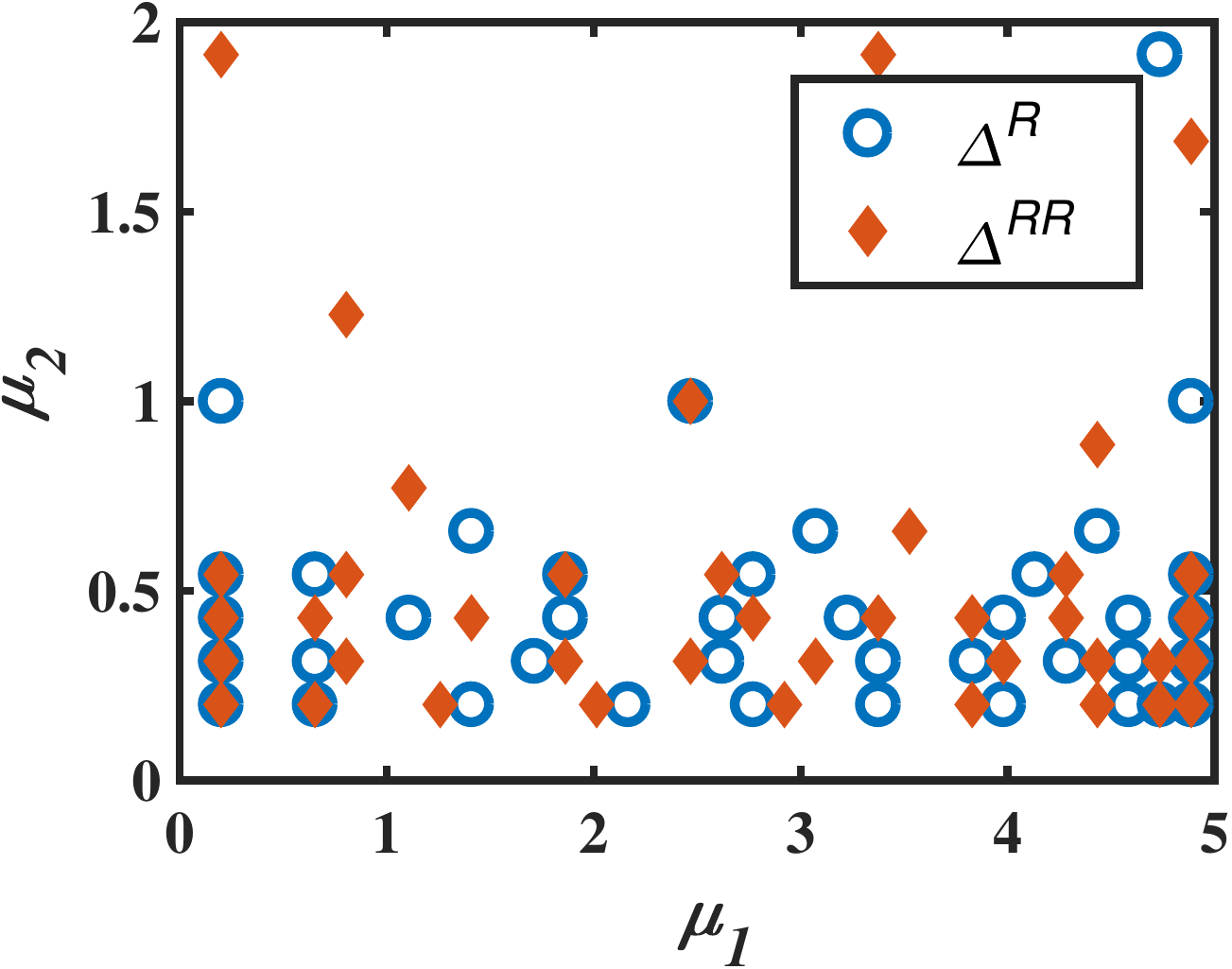}
\includegraphics[width=0.335\textwidth]{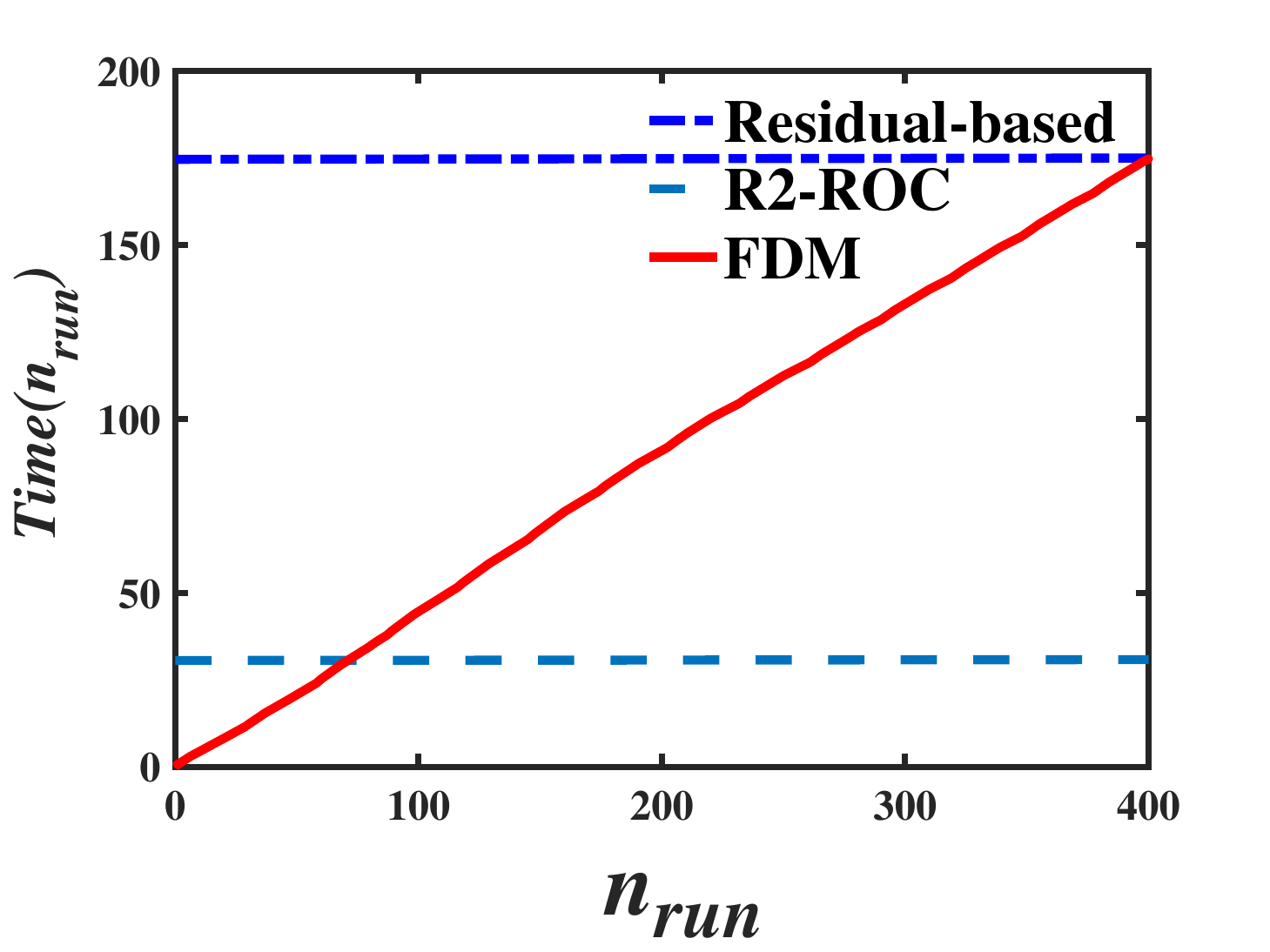}\\
 \includegraphics[width=0.49\textwidth]{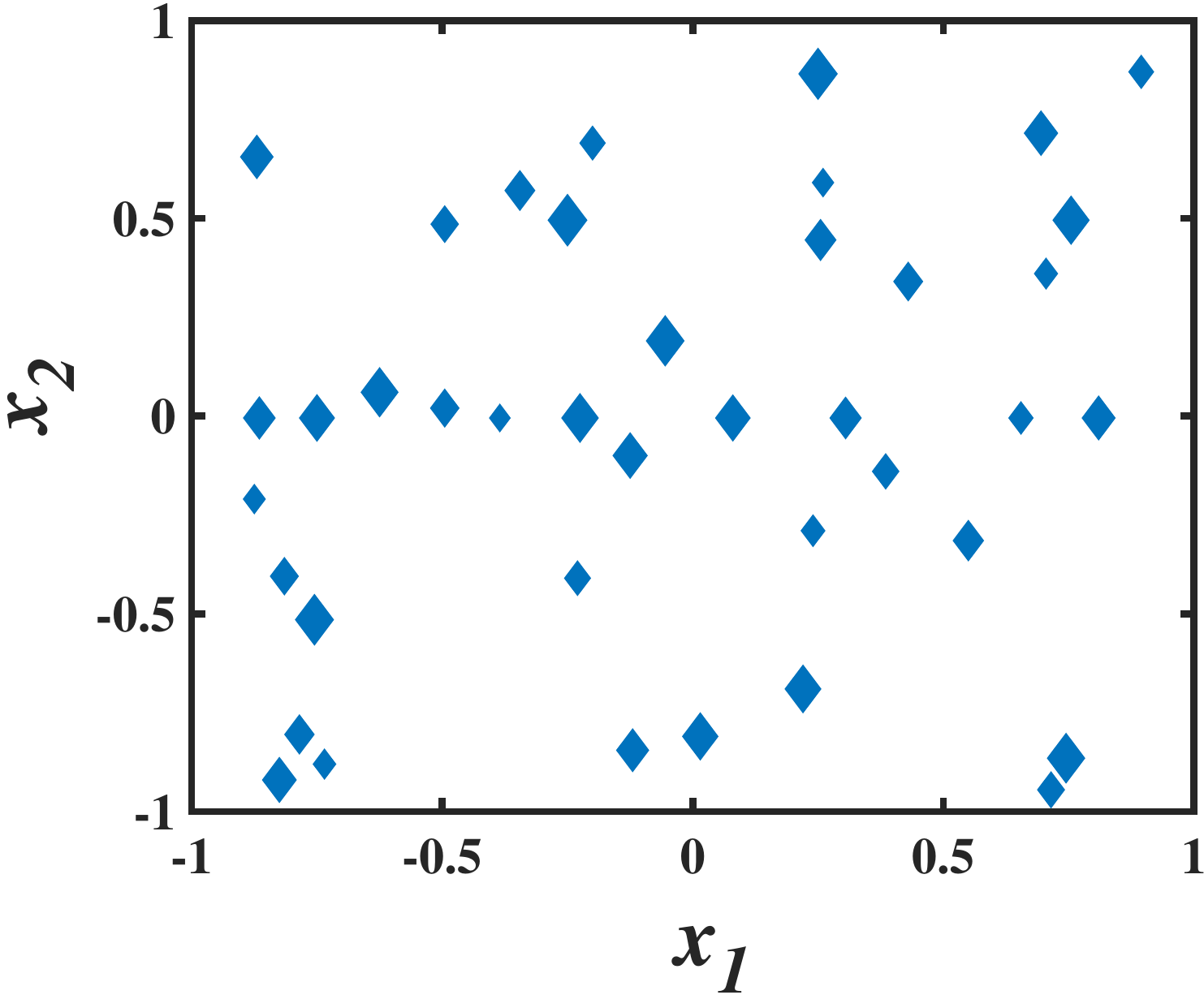}
\includegraphics[width=0.49\textwidth]{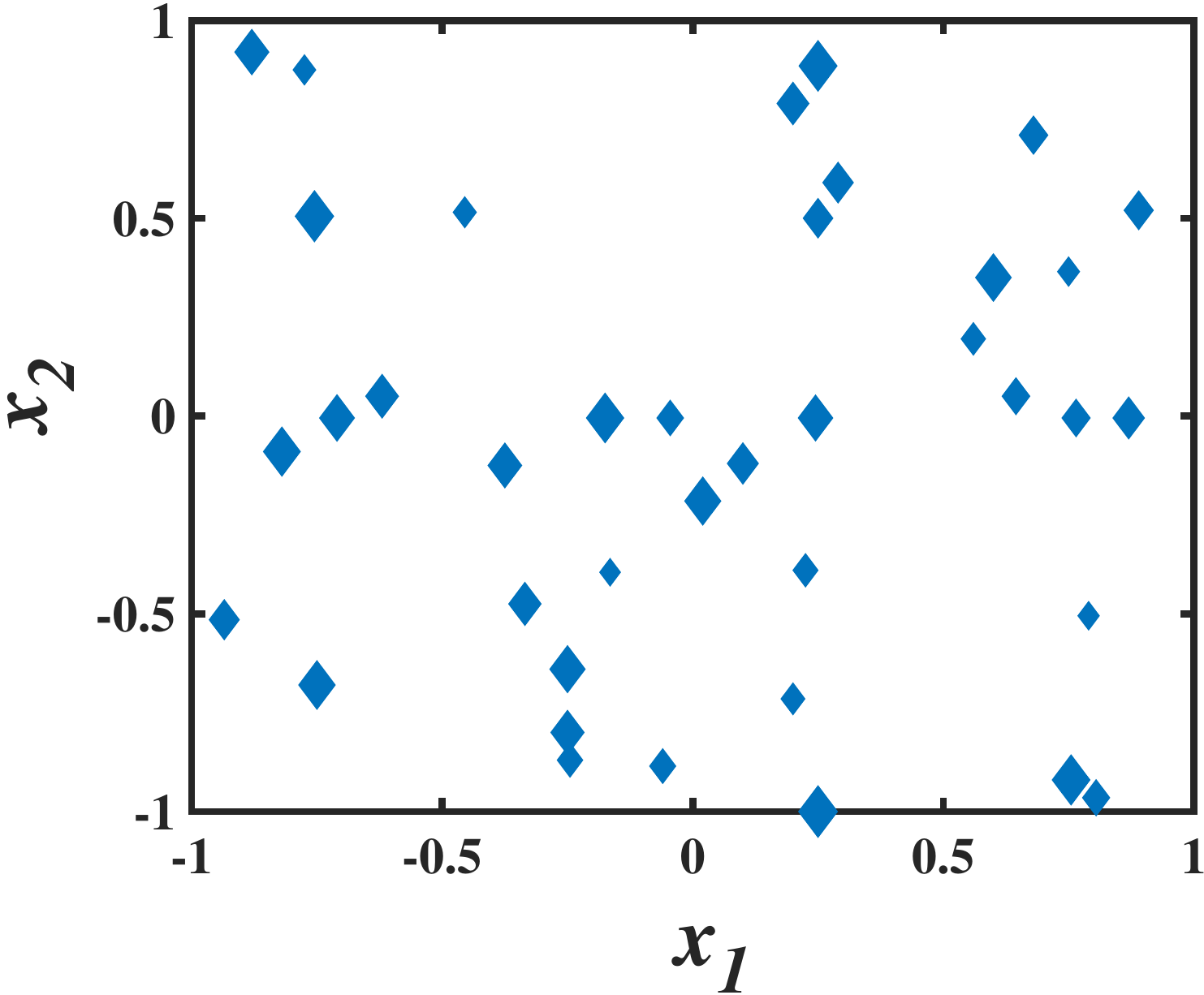}
\caption{Top row:(Left) comparison of the histories of convergence with $\sqrt{\calN} = 400$ for the errors and the error estimator for the ROC method. Here, $E^R$ and $E^{RR}$ refer to the $E(n)$ in \eqref{eq:error:steadyburger} with the reduced solution $\widehat{u}_n$ constructed by following the residual-based error estimator $\Delta^R$ and R2-based error estimator $\Delta^{RR}$, respectively. (Middle) Selected $N(=40)$ parameters of the ROC method for residual-based and R2-based approaches. (Right) cumulative runtime of the FDM, the residual-based, and R2-based RBM. Bottom row: selected $40$ collocation points $X^M_s$ from solutions (Left) and $39$ collocation points $X^M_r$ from residuals (Right).}
\label{2relativeerror}
\end{figure}

\begin{table}[!htb]
	\centering
		\begin{tabular}{ccccc}
		\hline
$(\mu_1, \mu_2)$ &~~$K$~~&Residual-based ROC &  ~ R2-ROC~ &~Direct FDM~~~~~  \\ \hline
\multirow{3}{*}{$(4.55, 0.42)$} &200	&0.003150  &0.004781 &2.310034 \\ 
&400	&  0.003067  &0.003931 & 11.779558 \\ 
&800	&  0.003258  &0.004185 &53.727031 \\ \hline
\multirow{3}{*}{$(1, 1.82)$} &200	&0.001125  & 0.001416 &0.662095 \\ 
&400	& 0.001141 & 0.001299 &3.338956  \\ 
&800	&0.001207  &0.001732 & 15.173460 \\ \hline
	\end{tabular}
	\caption{Online computational times (seconds) with different grid sizes $K$, when $N=40$. }
	\label{time2}
\end{table}

\subsubsection{Numerical comparison with POD and random generation} 

To further establish numerically the reliability of the R2-ROC algorithm, we compare it with two alternative methods of building the reduced basis space. On one end, POD \cite{BerkoozHolmesLumley1993, Kunisch_Volkwein_POD, WillcoxPeraire2002, LiangPOD} based on an exhaustive selection of snapshots (i.e. we include all solutions $u^\N(\bmu)$ for $\bmu \in \Xi_{\rm train}$)  produces the best reduced solution space and thus the most accurate, albeit costly, surrogate solution. We note that this version of POD only serves as reference and is in general not feasible as the full solution ensemble must be generated.  On the other end, a random selection of $N$ parameters as our RB snapshots is a fast but crude method. 
Comparison results of two steady-state test problems above are shown in Figure \ref{1compareerror}. Not surprisingly, the exhaustive POD is the most accurate. 
Our R2-ROC is one order of magnitude worse than POD, but in fact slightly better or comparable to the the best possible random generation. It is roughly one order of magnitude better than the median performance of random generations. 

\begin{figure}[!htb]
\centering
\includegraphics[width=0.49\textwidth]{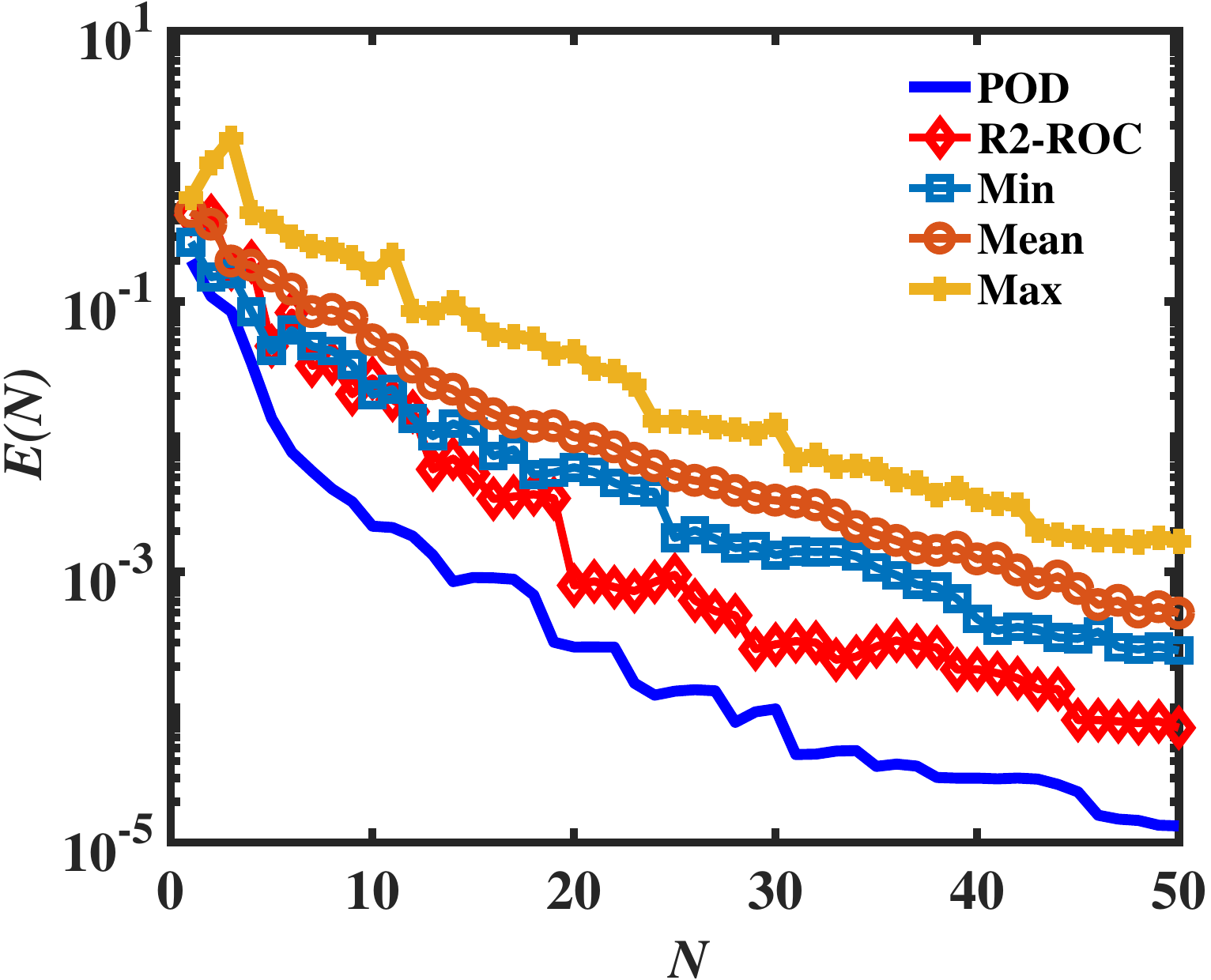}
\includegraphics[width=0.49\textwidth]{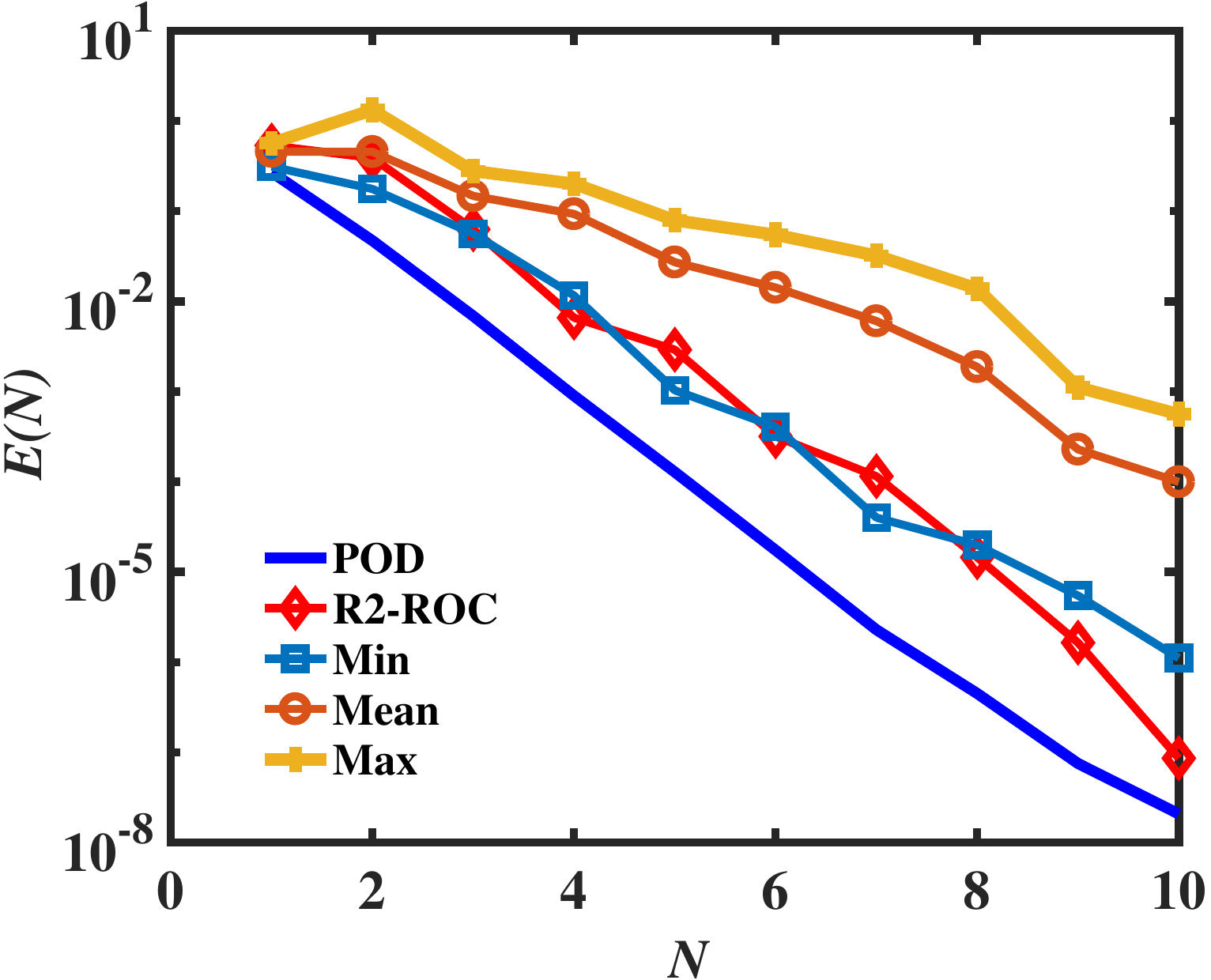}
  \caption{Convergence comparison for the R2-ROC, exhaustive POD and (best, median, and worst cases of) random generation approaches. On the left is for cubic reaction diffusion \eqref{eq:CRD} with $\sqrt{\calN}=400$, with the right being for steady viscous Burgers' equation \eqref{eq:burgers} with $\calN=100$. 
  }
\label{1compareerror}
\end{figure}

\subsection{Time dependent nonlinear problems}
\label{numerics:timedep}
In this section, we test the time-dependent equations corresponding to stationary problems in the last section, namely viscous Burgers' and cubic reaction diffusion equations.

\subsubsection{Viscous Burgers' equation}

We test the viscous Burgers' equation adopting settings similar to \cite{peherstorfer2019sampling, nguyen2009reduced}
\begin{equation}
\begin{split}
u_t + u u_x & = \mu u_{xx} + f(x),~ (x, t, \mu) \in (0,1) \times (0,1] \times \calD,\\
u(x,t=0; \mu) & =0,\\
u(0, t; \mu) = \alpha, \,\, & \,\, u(1, t; \mu) = \beta.
\end{split} 
\end{equation} 
The authors of \cite{peherstorfer2019sampling} takes $\calD = [0.1,1], f = 0, T = 1, \Delta t =10^{-4}, (\alpha, \beta) = (-1, 1)$ and monitor the average error in a Frobenius norm-based metric,
\begin{align*}
  \textrm{Error} = \frac{1}{m_{test}}\sum_{i=1}^{m_{test}} \frac{||u(\cdot,\cdot;\bmu) -\widehat{u}(\cdot,\cdot;\bmu)||_F}{||u(\cdot,\cdot;\bmu)||_F},~~ 
  \| v(\cdot,\cdot) \|^2_F &\coloneqq \sum_{\bx \in X^\N, t_i \in {\mathcal T}_f} v(\bx,t_i)^2
\end{align*}
while the authors in \cite{nguyen2009reduced} set $\calD = [0.005,1], f = 1, T = 2, \Delta t = 2 \cdot 10^{-6}, (\alpha, \beta) = (0,0)$ and observe the error in $L^2$. We investigate R2-ROC results from both of these setups. The results are showed in Figure \ref{fig:ben}. These results are similar to those of \cite{peherstorfer2019sampling, nguyen2009reduced}. When $N=10$, R2-ROC  attains an accuracy around $10^{-1}$ which gets much better when $N=15$.

\begin{figure}[!htb]
\centering
\includegraphics[width=0.49\textwidth]{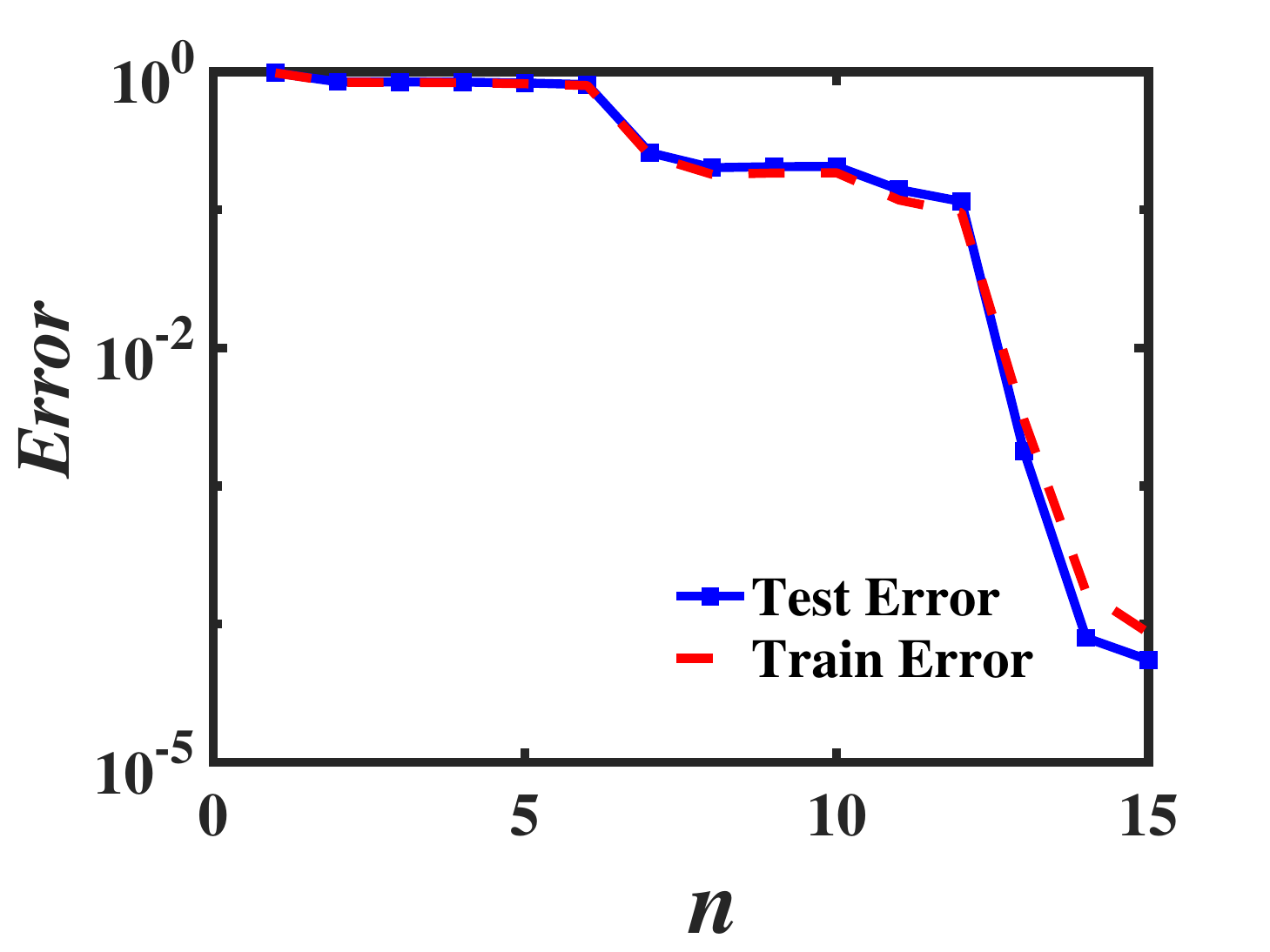}
\includegraphics[width=0.49\textwidth]{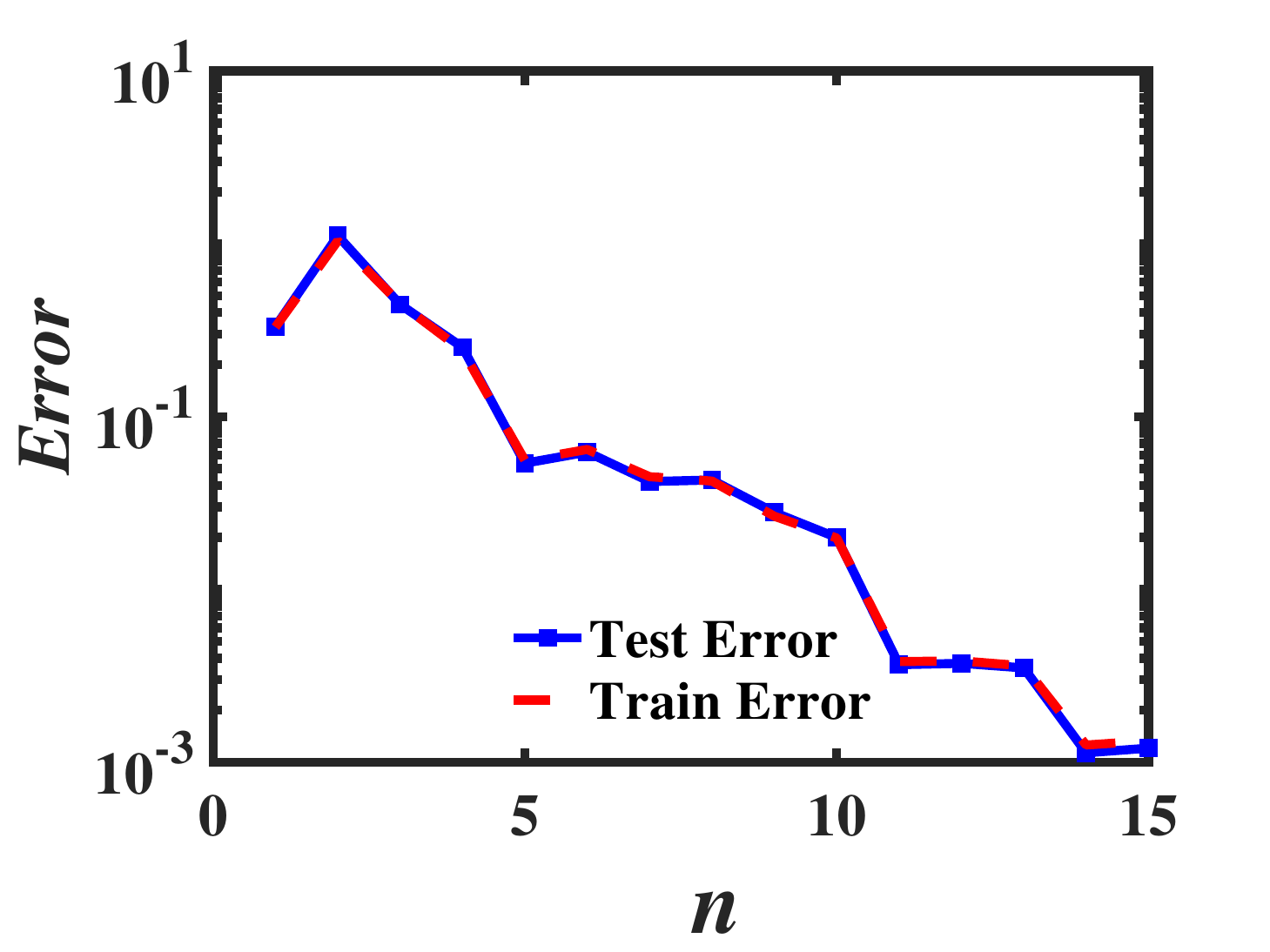}\\
\includegraphics[width=0.32\textwidth]{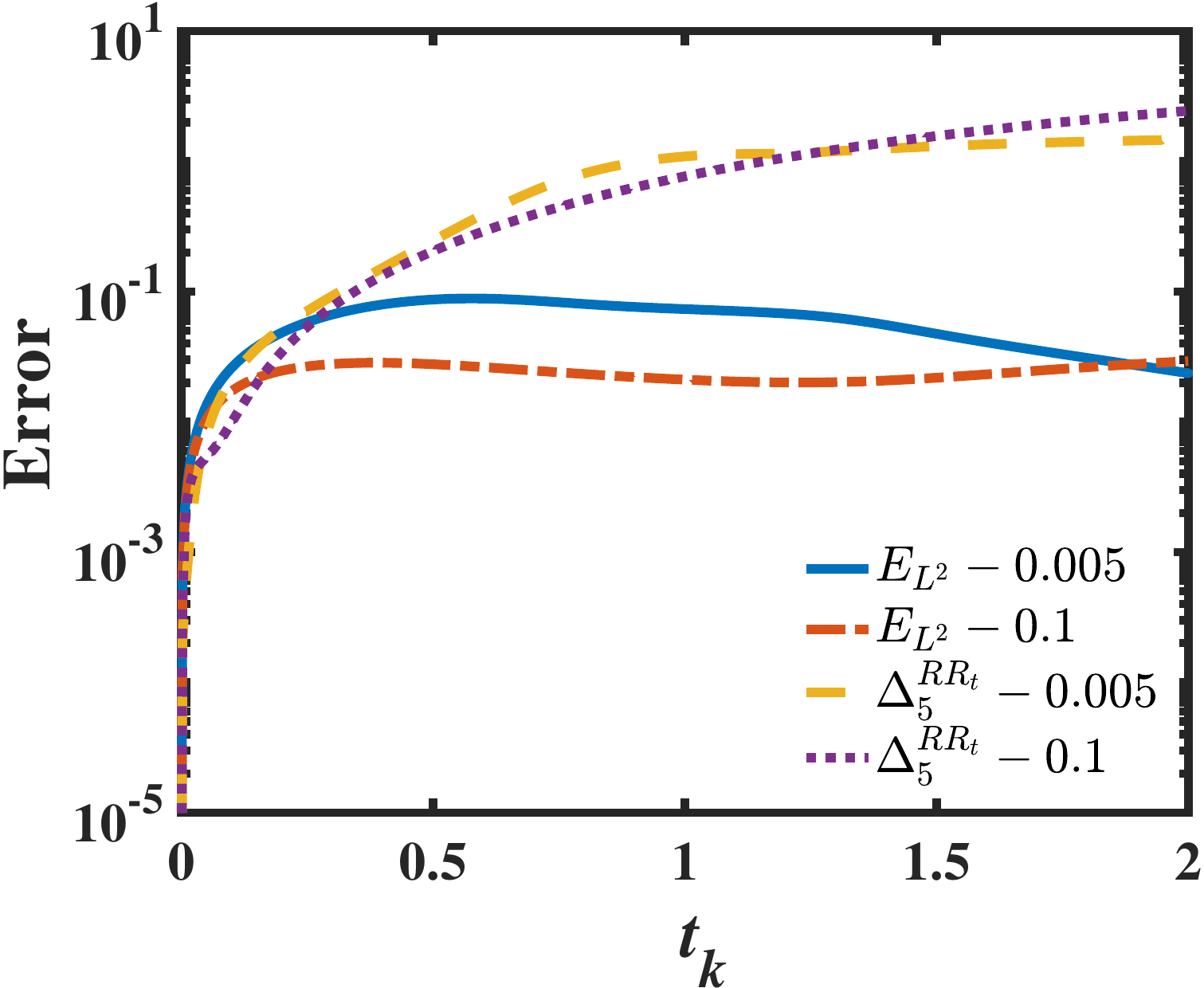}
\includegraphics[width=0.32\textwidth]{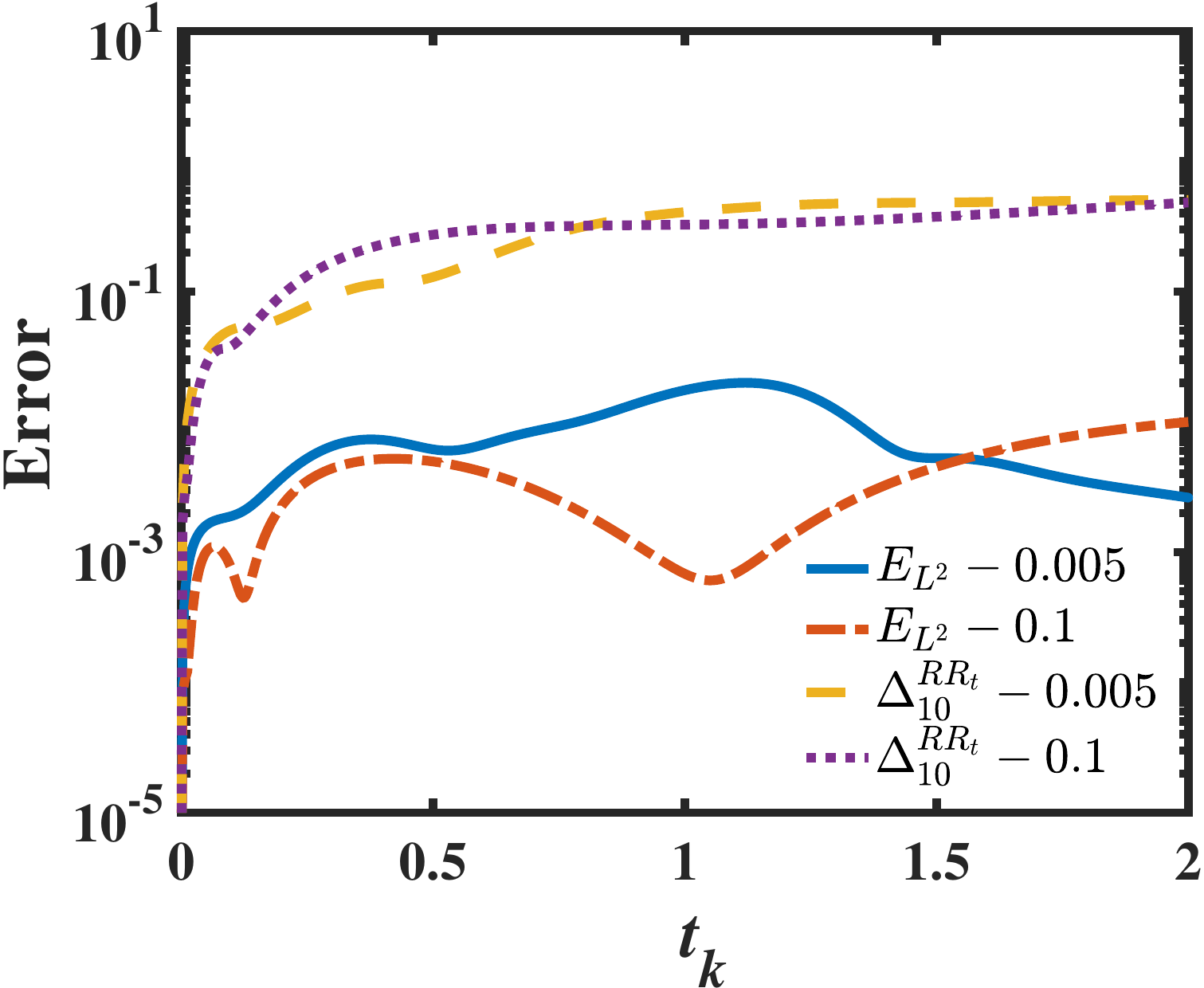}
\includegraphics[width=0.32\textwidth]{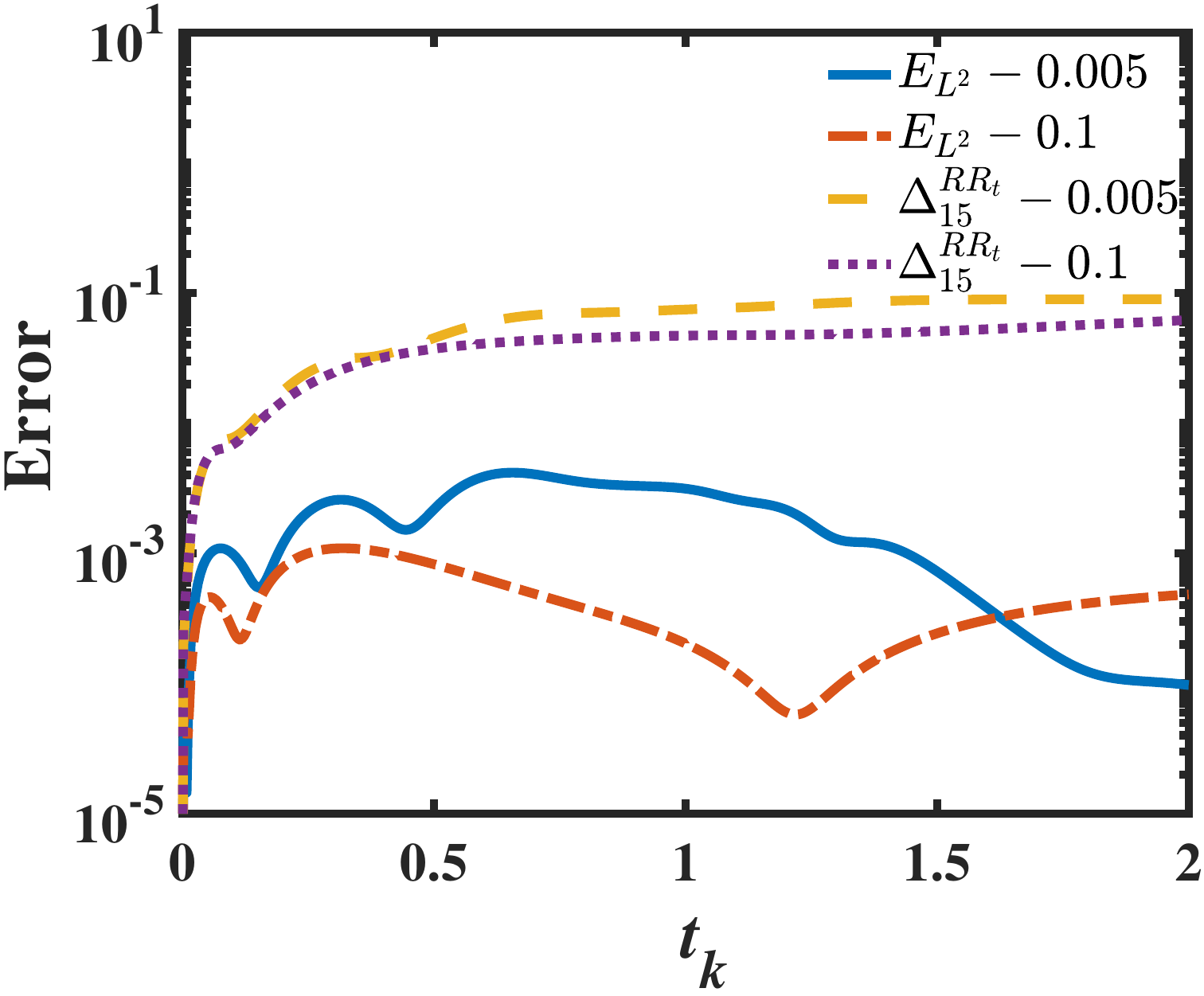}
\caption{
Transient viscous Burgers' result. On the top row are the error curves of R2-ROC with $N=15$ basis elements for the setup in \cite{peherstorfer2019sampling} (left) and \cite{nguyen2009reduced} (right). Plotted at the bottom row are the actual $L^2$ error, $||u^\N(:,t_k;\bmu) -u_N(:,t_k;\bmu)||$ and error estimator $\Delta_n^{RRt}(t_k;\bmu)$ as a function of discrete time $t_k$. The left, center and right plots show $N=5, 10, 15$, respectively, with parameter values bring $\bmu = 0.005,0.01, 0.1$ and the setup as in \cite{nguyen2009reduced}. }
\label{fig:ben}
\end{figure}

\subsubsection{Nonlinear reaction diffusion problems}

Next, we consider accordingly the following time dependent nonlinear reaction diffusion equation,
\begin{equation}
\begin{split}
u_t -\mu_2 \Delta u +u{(u-\mu_1)}^2 & = f(\bx), \mbox{ in } \Omega=[-1,1]\times [-1,1],\\
u & = 0 \mbox{ on } \partial \Omega, \\
u(\bx, t = 0) & = u_0(\bx).
\end{split}
\end{equation}
Here $f(\bx)=100\sin(2\pi x_1)\cos(2\pi x_2)$, and $[\mu_1,\mu_2] \in \calD := [1,5]\times [0.2,1]$. 
The parameter space $\calD$ is discretized by a $128 \times 32$ uniform tensorial grid. Denoting the step size along the $\mu_1$ direction by $h_1$, and the other by $h_2$, we specify the training and test sets as follows, 
\begin{align*}
  \Xi_{\rm train} &=  (1:8h_1:5) \times (0.2:2h_2:1), \\
  \Xi_{\rm test} &=   ((1+2h_1):4h_1:(5-2h_1)) \times ((0.2+h_2):4h_2:(1-h_2)).
\end{align*}
For the truth approximation, we use backward Euler for time marching and the same nonlinear spatial solver as the steady-state case \eqref{Operator2}.

Exponential convergence is evidenced in Figure \ref{fig:timecubic} top left. We also report the $\bmu$-component of  the parameter values selected by R2-ROC in  the top middle. Note that the RB space is built from the snapshots
\[
\left\{u(t^1_{\bmu^n}, \cdot; \bmu^n), \dots, u(t_{\bmu^n}^{k_{\bmu^n}}, \cdot; \bmu^n)\right\}_{n=1}^N.
\]
That is, for each distinct parameter value $\bmu^n$ chosen by R2-ROC, there are $k_{\bmu^n} \ge 1$ time level  snapshots $\{t_{\bmu^n}^1, \dots, t_{\bmu^n}^{k_{\bmu^n}}\} \subset \{t_0, t_1, \dots, t_{\calN_t}\}$. The red number by each $\bmu$ values in the middle pane denotes this $k_{\bmu^n}$. It is interesting to note that, consistent with the tendency of RBM selecting boundary values of the parameter domain, our R2-ROC tends to select multiple snapshots along time for the selected parameters when they are at the boundary of the parameter domain. 

To show the vast saving of the offline time for the R2-ROC approach, we present the comparison in cumulative computation time for the L1-ROC, R2-ROC, and the {high fidelity truth approximations} in Figure \ref{fig:timecubic} top right. 
We observe that the ``break-even'' number of runs for R2-ROC is smaller than that of the L1-ROC which is much smaller than that of the full simulation \cite{ChenJiNarayanXu2020}. The fact that they are even less than the dimension of the RB space underscores their efficiency. 
The bottom row of Figure \ref{fig:timecubic} shows the collocation points in the physical domain.

\begin{figure}[!htb]
\centering
\includegraphics[height = 0.20\textheight]{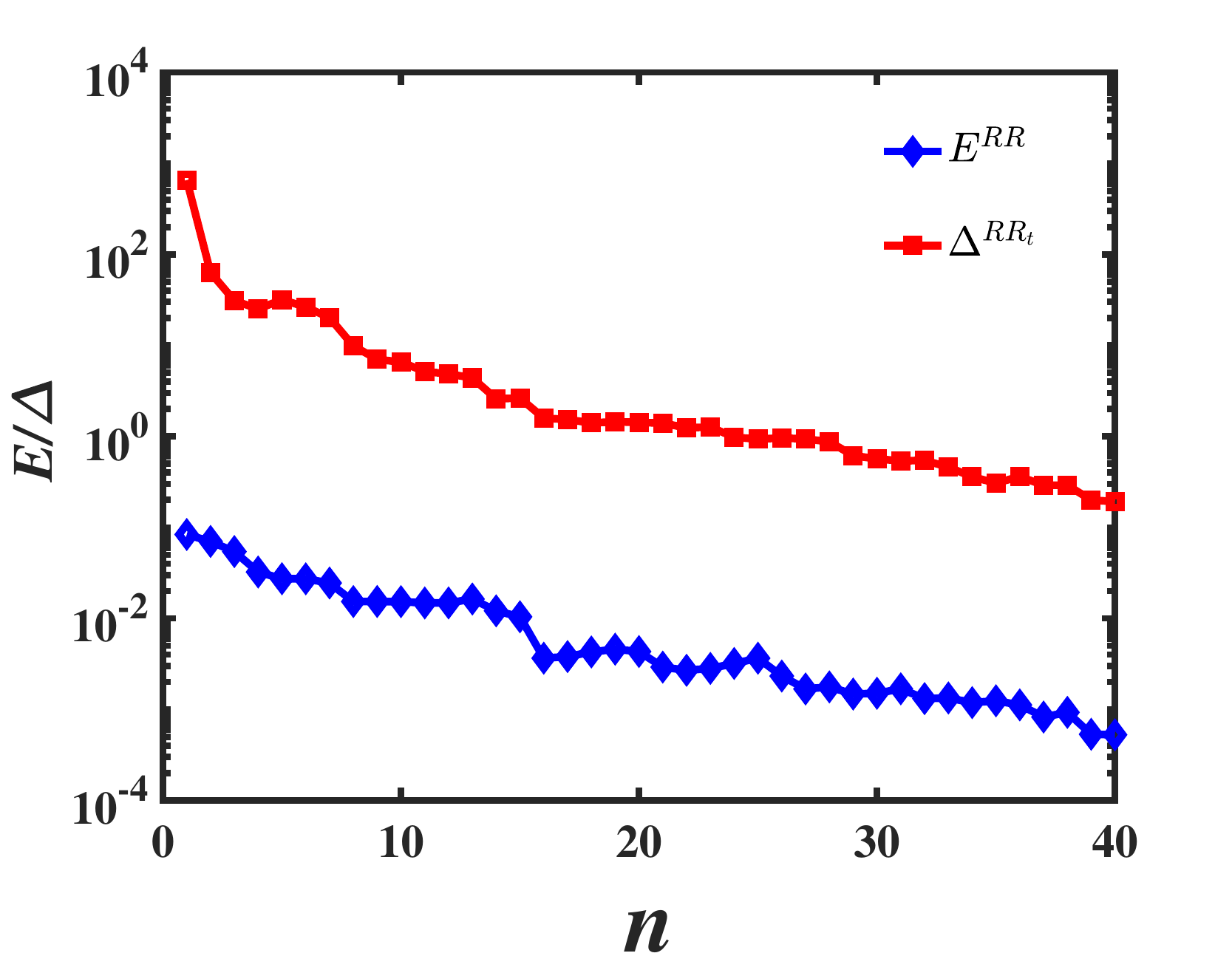}
\includegraphics[height = 0.20\textheight]{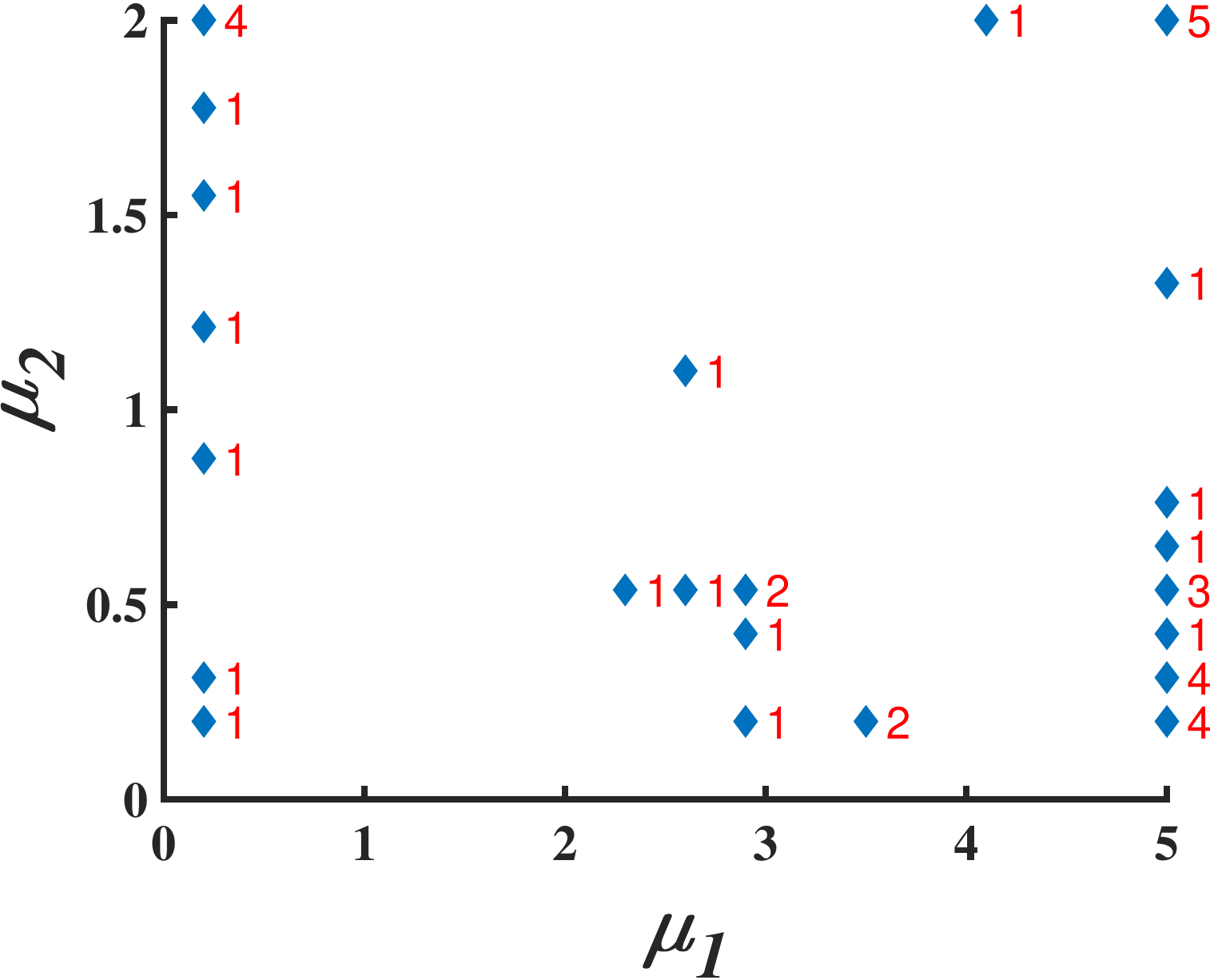}
\includegraphics[height = 0.20\textheight]{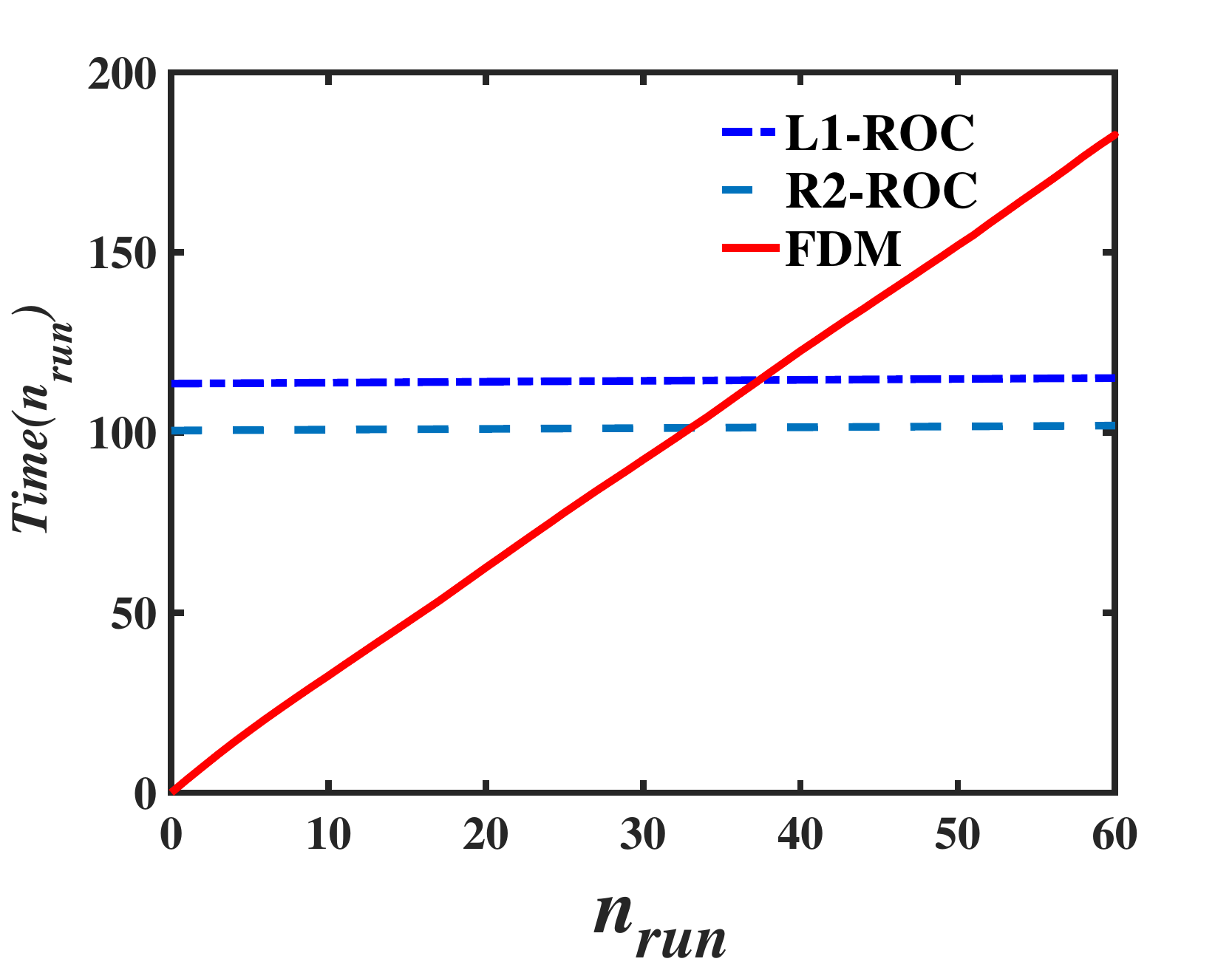}\\
\includegraphics[width=0.49\textwidth]{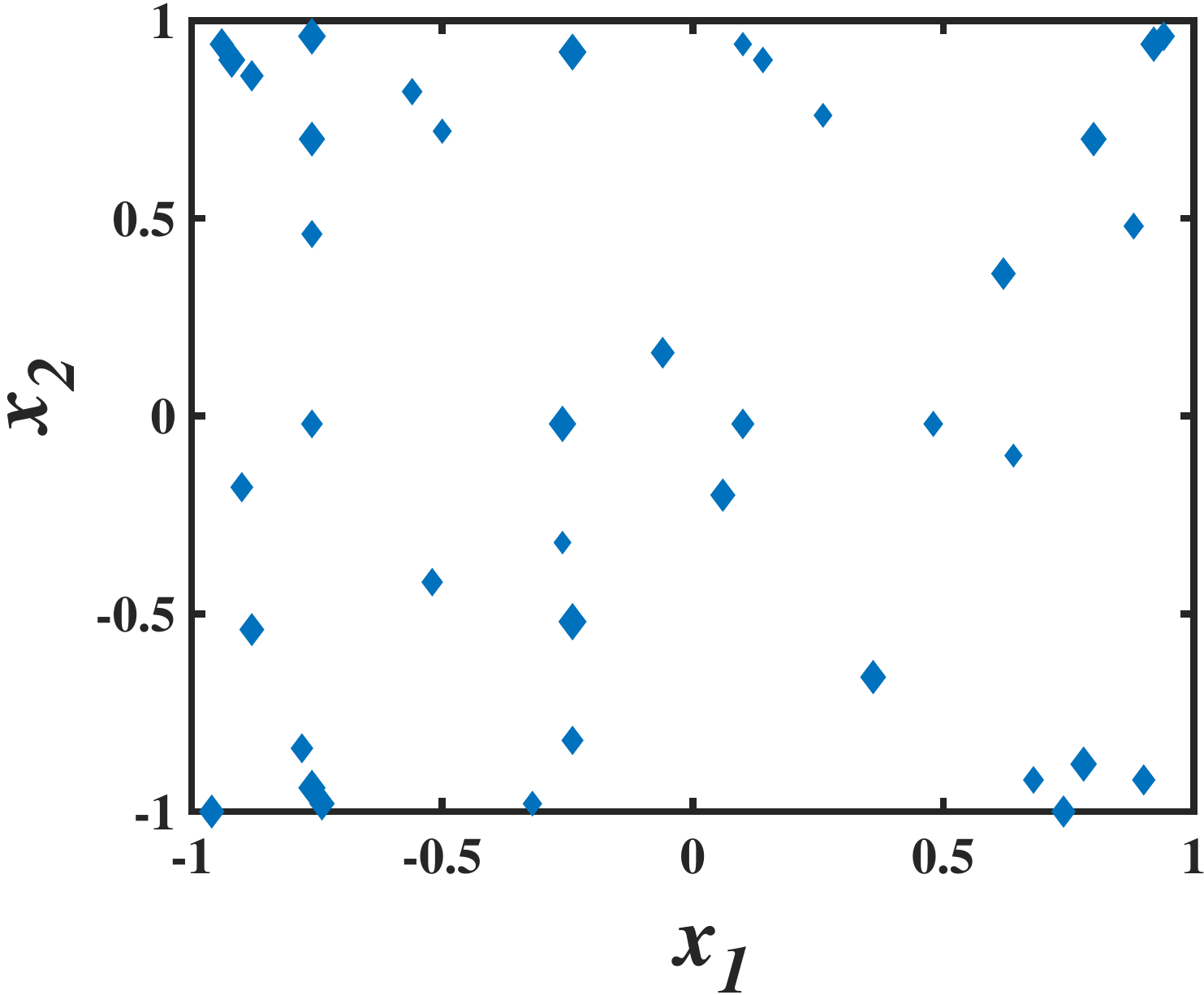}
\includegraphics[width=0.49\textwidth]{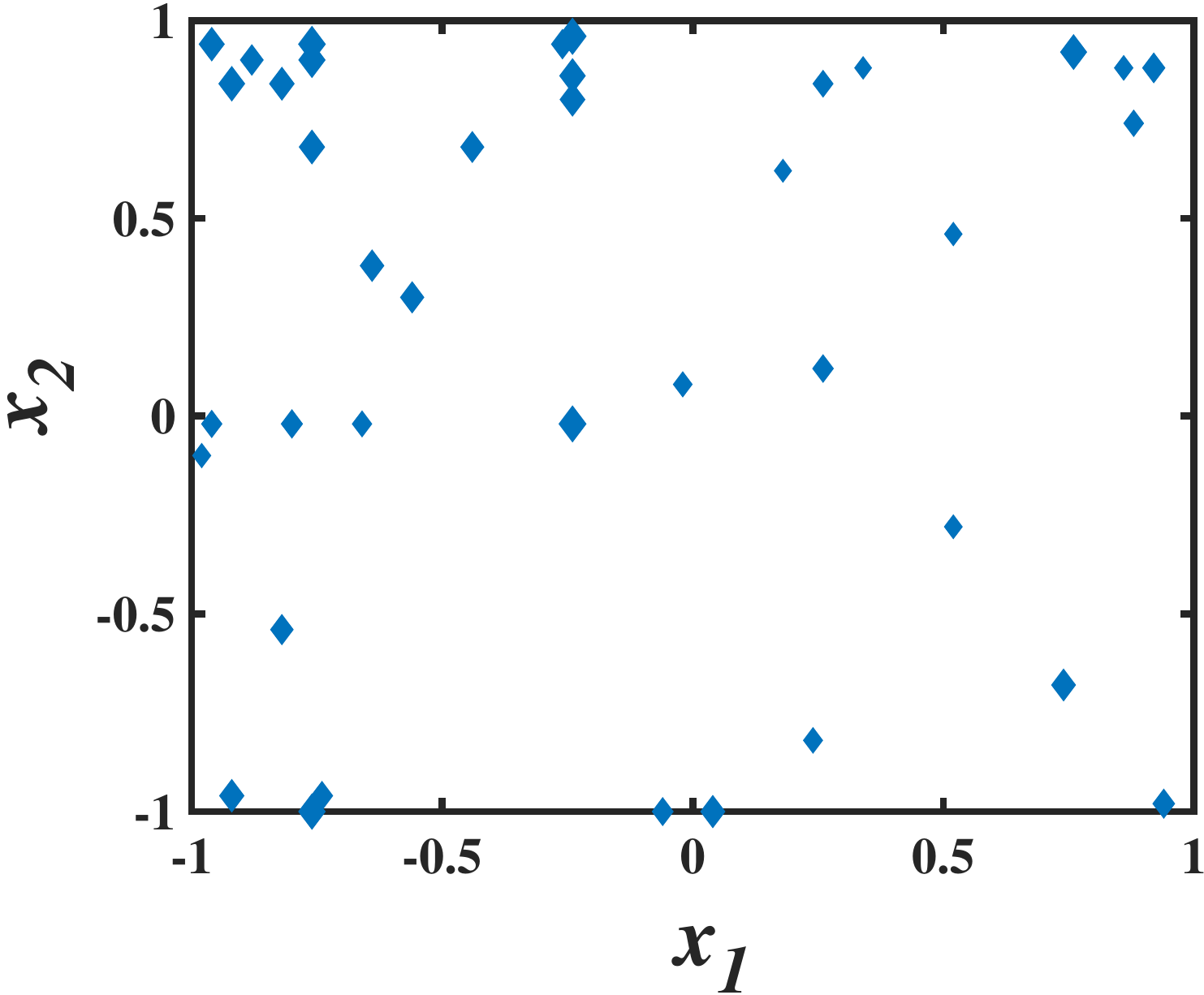}
\caption{Transient cubic reaction diffusion result. Top Left: Error curves of R2-ROC algorithm. Top Middle: Selected parameters when $N_{max}=40$. The number means corresponding parameter is selected at many different time nodes. Top Right: Cumulative run time comparison. Collocation points from solutions and residuals are shown at the bottom row from left to right respectively.  }
\label{fig:timecubic}
\end{figure}

\section{Conclusion}
\label{sec:conclusion}
This paper proposes a novel reduced over-collocation method, dubbed R2-ROC, for efficiently solving parametrized nonlinear and nonaffine PDEs. By integrating EIM/GEIM techniques on the solution snapshots and well-chosen residuals, the collocation philosophy, and the simplicity of evaluating the hyper-reduced well-chosen residuals, R2-ROC has online computational complexity independent of the degrees of freedom of the underlying FDM, and more interestingly, the number of EIM/GEIM expansion terms. 
This expansion would have otherwise significantly degraded the efficiency of a traditional RBM when applied to the nonaffine and nonliner terms in the equation. The lack of such precomputations of nonlinear and nonaffine terms makes the method dramatically faster offline and online, and significantly simpler to implement than any existing RBM. For future directions, we plan to extend R2-ROC from scalar to systems of nonlinear equations with nonaffine parameter dependence.

\bibliographystyle{abbrv}
\bibliography{rbmbib}

\end{document}